\documentclass{amsart}
\usepackage[utf8]{inputenc}
\usepackage{amsmath,amssymb,amsthm, graphics, comment,bm}
\usepackage{mathtools}
\usepackage{amsthm}
\usepackage{diagmac2}
\usepackage{textcomp}
\usepackage{color}
\usepackage[alphabetic]{amsrefs}
\usepackage[all,cmtip,color,matrix,arrow]{xy}
\usepackage{yfonts}
\usepackage{tikz}
\usepackage{tikz-cd}
\usepackage{stmaryrd}
\usepackage{mathrsfs}
\usepackage{enumerate}
\usepackage{hyperref}
\hypersetup{
	colorlinks=true,
	linkcolor=blue,
	filecolor=magenta,      
	urlcolor=cyan,
}
\usepackage[mathscr]{euscript}
\usepackage[all,cmtip]{xy}
\usepackage{cleveref}

\calclayout

\newcommand{\bA}{\mathbb{A}}

\newcommand{\bG}{\mathbb{G}}

\newcommand{\bP}{\mathbb{P}}
\newcommand{\bQ}{\mathbb{Q}}

\newcommand{\bZ}{\mathbb{Z}}

\newcommand{\km}{\mathfrak{m}}

\newcommand{\sI}{\mathscr{I}}

\newcommand{\oH}{\operatorname{H}}

\newcommand{\sP}{\mathscr{P}}

\newcommand{\sX}{\mathscr{X}}

\newcommand{\cA}{\mathcal{A}}
\newcommand{\cB}{\mathcal{B}}
\newcommand{\cC}{\mathcal{C}}
\newcommand{\cD}{\mathcal{D}}
\newcommand{\cE}{\mathcal{E}}
\newcommand{\cF}{\mathcal{F}}
\newcommand{\cG}{\mathcal{G}}
\newcommand{\cH}{\mathcal{H}}
\newcommand{\cI}{\mathcal{I}}
\newcommand{\cK}{\mathcal{K}}
\newcommand{\cL}{\mathcal{L}}
\newcommand{\cM}{\mathcal{M}}
\newcommand{\cN}{\mathcal{N}}
\newcommand{\cO}{\mathcal{O}}
\newcommand{\cP}{\mathcal{P}}

\newcommand{\cR}{\mathcal{R}}
\newcommand{\cS}{\mathcal{S}}
\newcommand{\cT}{\mathcal{T}}
\newcommand{\cU}{\mathcal{U}}
\newcommand{\cV}{\mathcal{V}}

\newcommand{\cX}{\mathcal{X}}
\newcommand{\cY}{\mathcal{Y}}
\newcommand{\cZ}{\mathcal{Z}}

\newcommand{\qB}{R}
\newcommand{\qR}{\cB}

\newcommand{\spec}{\operatorname{Spec}}

\newcommand{\sY}{\mathscr{Y}}

\newcommand{\Pic}{\operatorname{Pic}}

\newcommand{\Gm}{\bG_{{\rm m}}}
\newcommand{\GL}{\operatorname{GL}}

\newcommand{\Aut}{\operatorname{Aut}}
\newcommand{\PGL}{\operatorname{PGL}}

\newcommand{\Id}{\operatorname{Id}}

\newcommand{\bmu}{\bm{\mu}}
\newtheorem{theorem}{Theorem}[section]
\newtheorem{Teo}[theorem]{Theorem}
\newtheorem*{Teo*}{Theorem}
\newtheorem{Lemma}[theorem]{Lemma}

\newtheorem{Cor}[theorem]{Corollary}

\newtheorem{Prop}[theorem]{Proposition}
\newtheorem*{Ques*}{Question}

\theoremstyle{definition}
\newtheorem{Oss}[theorem]{Remark}
\newtheorem*{Oss'}{Remark}
\newtheorem{EG}[theorem]{Example}
\newtheorem{Def}[theorem]{Definition}
\newtheorem*{Def*}{Definition}
\newtheorem{Notation}[theorem]{Notation}
\newtheorem{Remark}[theorem]{Remark}

\begin{document}
\title[Degenerations of twisted maps to algebraic stacks]{Degenerations of twisted maps to algebraic stacks}
	\author[A. Di Lorenzo]{Andrea Di Lorenzo}
	\address[Andrea Di Lorenzo]{Dipartimento di matematica, Università di Pisa, Italia}
	\email{andrea.dilorenzo@unipi.it}
	\author[G. Inchiostro]{Giovanni Inchiostro}
	\address[G. Inchiostro]{University of Washington, Seattle, WA}
	\email{ginchios@uw.edu}
	\maketitle
	\begin{abstract}
		We give a definition of twisted map to a quotient stack with projective good moduli space, and we show that the resulting functor satisfies the existence part of the valuative criterion for properness.
	\end{abstract}
\section{Introduction}
Given a projective variety $X$, a classic problem is to study the geometry of curves in $X$. How many rational curves in $\bP^2$ of degree $d$ pass through $3d-1$ points? Or how many rational curves are there on a quintic threefold? These are two examples of famous questions, which were answered by Kontsevich \cite{Kon} by studying \textit{degenerations} of maps from curves to $X$. Indeed, Kontsevich constructed the so called moduli stack of stable maps, a generalization of the moduli stack of stable curves \cite{DM69}, which has since proved to be one of the main tools one can use to solve enumerative problems as the ones above (\cites{Kon, FP}).

With the recent developments in the theory of algebraic stacks \cites{Alp, halpner, AHH, luna}, it is natural to wonder what happens if one replaces the target projective variety $X$ with an algebraic stack $\cM$. For example, if the target algebraic stack is Deligne-Mumford and admits a projective coarse moduli space, the situation is similar to the case where $X$ is a projective variety. Abramovich and Vistoli (see \cites{AV_compactifying, ACV, AOV}) constructed a
space of so-called \emph{twisted} stable maps, that one can use to compactify the space of stable maps from a curve to $\cM$.

If the target is $\cB\Gm$ instead, studying degenerations of maps from a curve to $\cB\Gm$ is strongly related to the problem of finding a compactification of the universal Jacobian \cite{Cap} or, if the target stack is $\cB\GL_n$, a compactification of the universal moduli space of vector bundles over $\cM_{g}$ \cite{Pand}. Other approaches to these problems can be found in \cites{Kau, MT, FTT}.

The goal for this paper is to study degenerations of maps from a family of curves to an algebraic stack $\cM$ admitting a projective good moduli space $\cM\to M$, with the idea of interpolating between the stable maps of Kontsevich and Caporaso's compactified Jacobian. We introduce in \Cref{subsection twisted maps def} a notion of family of \textit{twisted} maps. Our main result can be formulated informally as follows.
\begin{Teo}[\Cref{theorem existence part val criterion}]\label{main theorem intro}
 Let $\cM:=[X/G]$ be a quotient stack by a linear group $G$ with a projective good moduli space, and assume we are given a family of twisted maps $\phi:\cC^* \to \cM$ over the punctured disc $\mathbb{D}^*$. Then, up to replacing $\mathbb{D}^*$ with a ramified cover, we can extend $\phi$ to a family of twisted maps $\cC\to \cM$ over the whole disk $\mathbb{D}$. 
\end{Teo}
\begin{Remark}
{There are two possible ways to regard our main result: on one hand, we prove that the functor of twisted maps to a quotient stack having a projective good moduli space satisfies the existence part of the valuative criterion for properness, and hence might constitute a first step towards the construction of a proper stack of stable maps to such kind of stacks.
Alternatively, our main result can also be regarded as a semistable reduction theorem, over one-dimensional bases, for maps from twisted curves to quotient stacks with projective good moduli spaces.} 
\end{Remark}
As already said, our definition of twisted map interpolates between the quasistable curves of Caporaso and the twisted stable curves of Abramovich and Vistoli. We show in \ref{subsection stacky nodes are necessary} and \ref{subsubsection P1s are necessary} that our definition is minimal: one cannot drop neither the stacky structure nor the destabilizing $\bP^1$'s and still hope for an analogue of \Cref{main theorem intro} to be true.

The main contribution of this work, beyond \Cref{main theorem intro}, is the explicit nature of our algorithm for constructing the degeneration of \Cref{main theorem intro} around each destabilizing $\bP^1$ (see Lemma \ref{lemma blow up + covering stack does the job}). 
For example, in \Cref{section example} we study some limits obtained from our algorithm in the case where target stack is $[\bA^1/\Gm]$ (namely, the stack parametrizing line bundles with section). In our example, we argue that one can assume the chain of destabilizing stacky $\bP^1$'s to have length at most 1, to map to the closed point of $[\bA^1/\Gm]$, and we explicitly describe the line bundle on the family of curves. We expect a similarly accessible description to be available in general. In some cases, in \Cref{theorem existence part val criterion} we provide an upper-bound to the maximal length of a chain of destabilizing $\bP^1$s.

\subsection*{Where are we headed}
As it is, our definition of twisted maps does not lead to a separated or bounded moduli space, essentially because we are imposing no numerical stability conditions.
For instance, when the target stack is $\cB\Gm$, we can ask for the line bundles to satisfy Caporaso's basic inequality \cite{Cap}*{0.3}, a numerical condition involving the degrees of the line bundles on the irreducible components of the curve. This is the key ingredient for making the moduli problem an algebraic stack with a projective good moduli space.

However, our definition of twisted maps can be regarded as a first step towards the construction of a stack of twisted \textit{stable} maps to a quotient stack having a projective good moduli space.
The next step in our project is to tackle the case $\cM=[X/\Gm^n]$, leveraging our
understanding of the behaviour of degenerations of twisted maps. In particular, we aim to show that, after imposing some numerical conditions, we actually get an algebraic stack which is $\Theta$-reductive and has unpunctured inertia: applying the results of Alper, Halpern-Leistner and Heinloth \cite{AHH}*{Theorem A}, we would get a good moduli space of twisted stable maps to $[X/\Gm^n]$. 

Another interesting direction to study is when $\cM=\cB G$ for a linear algebraic group $G$; in this case our functor of twisted maps can be regarded as the functor of $G$-torsors over families of curves, and our main theorem shows how to extend these families of $G$-torsors over curves. Hopefully, after imposing some numerical stability conditions, one can extract a proper algebraic stack out of our functor; we are currently investigating this path too.

\subsection*{Organization of the paper} In \Cref{section def} we introduce the main object of interest, namely twisted maps to an algebraic stack, and we present some basic properties of them. For example, we show that our definition recovers the one of Abramovich and Vistoli if $\cM$ is Deligne-Mumford. 

In \Cref{sec:fix} we show how to extend maps to codimension one points of the closed fiber of the family of curves (\Cref{teorema fix}).

In \Cref{section technical results} we introduce the main tools to extend maps to codimension two points: we prove two technical propositions (\Cref{lemma smooth purity} and \Cref{proposition val criterion for the nodes}) concerning extensions of line bundles over stacky surfaces, that will be essential for proving the main results of this section. 

In \Cref{section proof val criterion} we prove our main result (\Cref{main theorem intro}), conditional to a contraction theorem that we prove in the next section.

Indeed, in \Cref{subsection contracting rational curves}, we prove some results on the birational geometry of stacky surfaces, which are of independent interest: we study when one can contract certain (stacky) $\bP^1$s on a surface. Our results can be understood as a generalization of the results of Artin in \cite{Artin_contracting_curves} to stacky surfaces. 

Finally in \Cref{section example} we run a few examples, including the case of twisted maps to $\cM=\cB G$ and to $\cM=[\bA^1/\Gm]$.
\subsection*{Conventions} 
In what follows, unless otherwise stated, every stack and morphism is assumed to be essentially of finite type, with affine diagonal and over a field $k$ of characteristic zero. If $\cC$ is a one dimensional Deligne-Mumford stack with coarse moduli space $\cC\to C$, and $\cL$ is a line bundle on $\cC$, there is a minimal natural number $n$ such that $\cL^{\otimes n}$ descends to a line bundle $L$ on $C$. The \textit{degree} of $\cL$ will be $\frac{\operatorname{deg}(L)}{n}$. 

\subsection*{Acknowledgements} We are thankful to Dan Abramovich, Jarod Alper, Dori Bejleri, Samir Canning, Andres Fernandez Herrero, Daniel Halpern-Leistner, Ming Hao Quek, Mauro Porta, Zinovy Reichstein, Minseon Shin and Angelo Vistoli for helpful conversations. We thank the anonymous referees for carefully reading this paper and for providing detailed feedback.  

\section{Twisted maps to algebraic stacks}\label{section def}
In this section we first define twisted maps with target an algebraic stack $\cM$ having a projective good moduli space (\Cref{def twisted map}). Then we motivate our definition, showing that it recovers the one for Deligne-Mumford stacks \cites{AV_compactifying}. We finally provide two examples showing that our definition is somehow ``minimal", i.e. both stacky nodes and destabilizing $\bP^1$s have to be included in the domain curve.

\subsection{Twisted maps}\label{subsection twisted maps def}
We begin with the notion of quasi stable maps.
\begin{Def}
Let $n\ge 0$ be an integer, let $(C;p_1,...,p_n)$ be a nodal curve together with $n$ smooth points and let $M$ be a projective variety with ample line bundle $L$. We say that $f:C\to M$ is \emph{quasi stable} if $f^*L^{\otimes 3}\otimes \omega_C(p_1+...+p_n)$ is nef, and the irreducible components where $f^*L^{\otimes 3}\otimes \omega_C(p_1+...+p_n)$ is not positive are isomorphic to $\bP^1$. Such components are called \emph{destabilizing $\bP^1$s}. A connected subcurve of $C$ consisting of a union of destabilizing $\bP^1$s 
will be called \textit{chain of destabilizing $\bP^1$s}.
\end{Def}
Recall \cite[Definition 4.1.2]{AV_compactifying} that a \emph{twisted nodal} $n$\emph{-pointed curve} over $B$ is the data of three morphisms
$\Sigma_i\to \cC\to C\to B$
such that 
\begin{itemize}
    \item $\cC$ is a Deligne-Mumford stack, the map $\cC\to B$ is proper and \'{e}tale-locally $\cC$ is a nodal curve;
    \item $\Sigma_i\subset \cC$ are $n$ disjoint closed substacks contained in the smooth locus of $\cC$ such that $\Sigma_i\to C \to B$ are \'{e}tale gerbes;
    \item $\cC\to C$ is a coarse moduli space and is an isomorphism away from $\Sigma_i$ and the singular locus of $C\to B$.
\end{itemize}
From now on, we will say \emph{twisted ($n$-pointed) curve} to indicate a twisted nodal $n$-pointed curve.
\begin{Def}\label{def twisted map} 
Let $\cM$ be an algebraic stack with projective good moduli space $\cM\to M$, let $n\ge 0$ be a non-negative integer. A \emph{twisted map to} $\cM$ \emph{over a scheme} $B$ consists of a triplet $(\pi:\cC\to B, \phi:\cC\to \cM,\{\sigma_i : \Sigma_i \to \cC\})$ such that:
\begin{enumerate}
\item $(\pi:\cC\to B,\Sigma_i)$ is a twisted $n$-pointed curve over $B$ and $\phi$ is representable;
\item if we denote by $C$ (resp. $S$) the coarse moduli space of $\cC$ (resp. $\cup \Sigma_i$), then for every $b\in B$ the pointed map $f_b:(C_b,S_b)\to M$ is quasi stable;
   \item if $\cD$ is an irreducible components of $\cC_b$ with coarse moduli space which is an $f_b$-destabilizing $\bP^1$, then for every $d$ the map $\cD\to\cM$ does not factor as $\cD\to \cB\bmu_d\to\cM$. 
    \end{enumerate}
\end{Def}

\begin{Oss}We will mainly deal with the case where $n=0$, namely the non-pointed case. The reason we introduce pointed quasi stable maps is because they arise naturally when taking the normalization of an unpointed quasi stable map.
\end{Oss}

We now comment briefly on the last point of \Cref{def twisted map}.
We begin by showing that if $\cM$ is a Deligne-Mumford stack, we recover the definition of twisted stable maps \`{a} la Abramovich-Vistoli. For this, we need the following two lemmas.
\begin{Lemma}\label{lemma C to p1 ramified at 2 points implies C=p1}
Let $f:C\to \bP^1$ be a generically \'etale morphism which is ramified only over $\{p,q\}\subseteq \bP^1$. Then $C\cong \bP^1$ and the map is $[a,b]\mapsto [a^n,b^n]$ for some $n$.
\end{Lemma}
\begin{proof}
First observe that $p\neq q$. Indeed, since the fundamental group of $\bP^1$ without a point is trivial, if $p=q$ then $f$ must be an isomorphism. 

Then $p\neq q$ and consider $\omega_{\bP^1}(p+q)$. A local generator for this line bundle around $p$ is $\frac{dx}{x}$.
Locally around every point $p_i$ in the preimage of $p$, $f$ can be written as $x\mapsto x^{e_i}$ for some integer $e_i$; therefore, the pullback of $\frac{dx}{x}$ is $\frac{dx^{e_i}}{x^{e_i}}=e_i\frac{dx}{x}$: we deduce that $h^*\omega_{\bP^1}(p+q)=\omega_C(R)$ where $R$ is the ramification locus of $h$.
But $\omega_{\bP^1}(p+q)$ has degree zero, so also $h^*\omega_{\bP^1}(p+q)=\omega_C(R)$ has degree zero: as $C$ is a smooth connected curve, this implies that $C\cong \bP^1$, $R$ consists of exactly two points, and $f$ is totally ramified. Then $f$ is of the desired form.
\end{proof}

\begin{Lemma}\label{lemma stabilizers stacky destabilizing p1s}
Let $\cP^1$ be a smooth and proper Deligne--Mumford stack having coarse moduli space $\bP^1$ and assume that the coarse moduli space map $\cP^1\to \bP^1$ is an isomorphism away from $0$ and $\infty$. Assume that there is a representable morphism $f:\cP^1\to \cB G$ for a finite group $G$. Then there is an integer $e$, a map $f':\cP^1\to \cB\bmu_e$ and an embedding $\iota:\cB\bmu_e\hookrightarrow \cB G$ such that $f=\iota \circ f'$.
\end{Lemma}
\begin{proof}
We use the following fact: given a $G$-torsor $P\to S$ over a connected scheme $S$ and a connected component $P_0\subset P$, the classifying morphism $S\to \cB G$ factors as $S\to \cB H\to \cB G$, where $H$ is the subgroup of elements of $G$ fixing $P_0$. 

Indeed, if $P'\subset P$ is another connected component, there exists an element $g\in G$ such that $gP_0\subset P'$: it is immediate to check that this induces an isomorphism $P'\simeq P_0$, hence $P\simeq \coprod_{gH \in G/H} P_0$. In particular, there is a classifying morphism $S\to \cB H$ such that $\spec(k)\times_{\cB H} S\simeq P_0$. On the other hand, the pullback of the universal $G$-torsor $\spec(k)\to \cB G$ along $\cB H\to \cB G$ is isomorphic to $\coprod_{gH \in G/H} \spec(k)$, from which we deduce that the pullback of the universal $G$-torsor along $S\to \cB H\to \cB G$ is isomorphic to 
\[ \coprod_{gH\in G/H} (\spec(k))\times_{\cB H} S \simeq \coprod_{gH \in G/H} P_0. \]
We can then conclude that the two classifying maps $S\to \cB G$ and $S\to \cB H\to \cB G$ are isomorphic.

Let $C'\to \cP^1$ be the $G$-torsor induced by $\cP^1\to \cB G$, and let $C$ be a connected component of $C'$. It is smooth as it is \'etale over $\cP^1$ and it is proper as $\spec(k)\to \cB G$ is proper. Therefore $C$ is a curve. 

We can take the composition $h:C\to \cP^1\to \bP^1$: this is a cover of $\bP^1$ ramified at at most two points. But then from \Cref{lemma C to p1 ramified at 2 points implies C=p1}, the map $h$ must be totally ramified at these two points with ramification index $e$.
We deduce that $h$ is given by the two sections $X_0^e, X_1^e$, the covering $C\to\bP^1$ is cyclic or, in other terms, is a $\bmu_e$-torsor for some integer $e$. By a descent argument, the subgroup $H$ of elements in $G$ that fix $C$ must be isomorphic to $\bmu_e$.

We can now apply the fact mentioned at the beginning to conclude that the map $\cP^1\to \cB G$ factors through $\cB\bmu_e$.\end{proof}
\begin{Prop}
Let $\cM$ be a Deligne-Mumford stack with coarse moduli space $M$, and let $\phi:\cC\to \cM$ be a twisted map over $\spec(k)$, with $f:C\to M$ the induced map on the coarse spaces. Then $f$ is Kontsevich stable and $\phi$ is a twisted stable map in the sense of Abramovich-Vistoli.
\end{Prop}
\begin{proof}Assume by contradiction that it is not, let $\cD$ be an irreducible component of $\cC$ with coarse moduli space $D\cong \bP^1$, and assume that $D$ is a destabilizing $\bP^1$.  Then $f|_D:D\to M$ factors as $D\to \spec(k)\to M$, where $\spec(k)\to M$ is the inclusion of a closed point $p$.  By the universal property of the fiber product, $\cD\to \cM$ factors through $\spec(k)\times_M\cM \cong \cB G$ where $G$ is the stabilizer of the closed point $p$.  By \Cref{lemma stabilizers stacky destabilizing p1s},  the map $\phi|_\cD:\cD\to \cM$ factors via $\cD\to \cB\bmu_e\to \cB G\to \cM$, which is a contradiction.
\end{proof}

\subsubsection{Destabilizing components and multidegree}
An interesting case is when $\cM = [\spec(A)/\Gm]$ for $\spec(A)$ an affine $\Gm$-scheme of finite type over $k$. In fact $\cM$ is an algebraic stack and its good moduli space is projective if and only if it is a point. In this case, we can consider twisted maps to $[\spec(A)/\Gm]$.
Observe that the composition $\cC\to [\spec(A)/\Gm]\to \cB\Gm$ corresponds to a line bundle $L$ over $\cC$. The following proposition shows that condition (3) of \Cref{def twisted map} can then be formulated in terms of the degree of $L$ on the destabilizing components.
\begin{Prop}\label{prop degree 0 = factors through bmuk}
Let $(\pi:\cC\to B,\phi:\cC\to [\spec(A)/\Gm])$ be a pair satisfying the first two conditions of \Cref{def twisted map}. Then it satisfies condition (3) (hence it is a twisted map) if and only if for every geometric point $b\in B$ the line bundle $L_b$ on $\cC_b$ induced by the composition $\cC_b\to [\spec(A)/\Gm]\to \cB\Gm$ has non-zero degree on every destabilizing component.
\end{Prop}
\begin{proof}
Suppose that for some geometric point $b\in B$ we have a destabilizing component $\cD_b\subset \cC_b$ such that the line bundle $L|_{\cD_b}$ has degree zero. Then the coarse moduli space of $\cD_b$ is isomorphic to $\bP^1$. As $\cD_b$ is a Deligne-Mumford stack, there exists an integer $e$ such that $L|_{\cD_b}^{\otimes e}$ comes from a line bundle on the coarse space $D_b\simeq\bP^1$. As the only line bundle on $\bP^1$ of degree zero is the trivial one, we deduce that $L|_{\cD_b}^{\otimes e}\simeq \cO_{\cD_b}$.

This implies that we have a factorization $\cD_b\to \cB\bmu_{e}\to \cB\Gm$, hence a map 
\[ \cD_b\longrightarrow \cB\bmu_e\times_{\cB\Gm}[\spec(A)/\Gm]\simeq [\spec(A)/\bmu_e]. \]
Observe now that the coarse space $D_b\simeq\bP^1$ is contracted to a point $q$ in the coarse space of $[\spec(A)/\bmu_e]$, because the latter is affine. This implies that $\cD_b\to [\spec(A)/\bmu_e] $ factors through some $\cB\bmu_{e'}$, where $\bmu_{e'}\subset \bmu_{e}$ is the automorphism group of the geometric point whose image in the coarse space $\spec(A^{\Gm})$ is $q$. We showed that condition (3) in \Cref{def twisted map} is not satisfied.

Viceversa, suppose that condition (3) is not satisfied: then there exists a factorization 
\[\cD_b \longrightarrow \cB\bmu_{e} \longrightarrow [\spec(A)/\Gm] \longrightarrow \cB\Gm. \]
This implies that the line bundle $L|_{\cD_b}$ comes from $\cB\bmu_{e}$, i.e. is a torsion line bundle.
\end{proof}

\subsection{Why twisted maps with quasi stable coarse space}
In this section we report two examples showing that if one allows neither twisted curves nor destabilizing stacky $\cP^1$s, then there is no hope to find a moduli space satisfying the existence part of the valuative criterion for properness.

\subsubsection{Stacky nodes are necessary}\label{subsection stacky nodes are necessary}
This example already appeared in the work of Abramovich and Vistoli. Consider the two homogeneous polynomials
$A\in \oH^0(\bP^1, \cO(4)) $ and $B \in \oH^0(\bP^1, \cO(6))$ such that $4A^3 + 27B^2$ has 12 distinct roots. This data corresponds to a Weierstrass fibration $(X,S)\to \bP^1$ with 12 nodal singular fibers, so it corresponds to
a map $\bP^1\to \overline{\cM}_{1,1}$.
We can now multiply all the terms in $A$ of degree greater than 2 and the terms in $B$ of degree greater than 3 by a parameter $t$. This gives two polynomials in $\oH^0(\bP^1_{k(t)}, \cO(4))$ and $\oH^0(\bP^1_{k(t)}, \cO(6))$ respectively. The corresponding Weierstrass fibration $(X_{k(t)},S_{k(t)})\to \bP^1_{k(t)}$ still corresponds to a map $\bP^1_{k(t)}\to \overline{\cM}_{1,1}$ which induces a map $\bP^1_{k(t)}\to \overline{\cM}_{1,1}\to \overline{\operatorname{M}}_{1,1}$. We can take the limit of this map to $\spec(k[t]_{(t)})$. It is
easy to see, from how we multiplied the coefficients of $A$ and $B$ by $t$, that if we denote by $C\to \overline{\operatorname{M}}_{1,1}$ the limiting map, the special
fiber of $C$ will have a component $D$ such that $D\to \overline{\operatorname{M}}_{1,1}$ has degree 6. Even if we replace $C$ with another scheme $C'\to C$ birational to $C$ (blowing up the closed fiber of $C\to \spec(R)$), the proper transform of $D$ in $C'$ will still map to $\overline{\operatorname{M}}_{1,1}$ with a map of degree 6. This cannot lift to $\overline{\cM}_{1,1}$: any map from a curve to $\overline{\cM}_{1,1}$ has degree divisible by 12.
\subsubsection{Destabilizing $\bP^1$s are necessary}\label{subsubsection P1s are necessary} Let $R$ be a DVR, and consider two genus 2 curves with two marked points, $\{(C_i;p_i,q_i)\}_{i=1}^2$, and consider the curve obtained by gluing together $p_1$ with $p_2$ and $q_1$ with $q_2$. Let $\cC=C\times\spec(R)$, and let $\eta$ be the generic point of $\spec(R)$. 
Consider the trivial line bundle over the two components of $\cC_\eta$, glued via the multiplication by a constant on the $p_i$s and by the multiplication by the uniformizer $\pi$ of $\spec(R)$ on the $q_i$.
This corresponds to a map $f:\cC_\eta \to \cB\Gm$. One can check that:
\begin{enumerate}
    \item the only twisted curve, with stable coarse moduli space, that one can produce as a limit of $\cC_\eta$ over $\spec(R)$ is the scheme $\cC$, and
    \item the map $f$ does not extend to $\cC\to \cB\Gm$ as the line bundle does not extend.
\end{enumerate}
In particular, if one does not allow destabilizing $\cP^1$s, one cannot fill in the limit with just twisted curves.

\section{Extending maps from deformations of curves I}\label{sec:fix}
Let $\cM=[X/G]$ be a quotient stack by a linear algebraic group $G$ having a projective good moduli space $M$. Let  $\qB$ be a DVR with generic point $\eta$, and let $\pi:C_\eta\to \eta$ be a smooth curve. Suppose we are given a morphism $\phi_\eta:C_\eta\to \cM$ inducing a map $f_\eta:C_\eta\to M$ which is Kontsevich-stable. Kontsevich's theorem tells us that there exists a unique family of Kontsevich-stable curves $C\to \spec(\qB)$ together with a map $f:C\to M$  that extends $f_\eta$. Can we lift $f$ to a map $\phi:C\to \cM$ that extends $\phi_\eta$? A first step towards solving this problem would be lifting $f$ {at least at the generic points of the closed fiber of }$C\to \spec(\qB)$. In this section we show that this is indeed the case, up to taking ramified covers of the base (see \Cref{teorema fix}).

In what follows by an extension of DVRs we mean a map of DVRs which is generically finite. 
\begin{Lemma}\label{lemma tensor product thing is a DVR}
Let $R$ be a DVR with algebraically closed residue field, and consider $ \phi\colon C\to \spec(R)$ a flat morphism of finite type, with pure 1-dimensional and smooth fibers. Let $\mu$ be the generic point of an irreducible component of the closed fiber of $\phi$, let $\cB:=\cO_{C,\mu}$ and let $f\colon R\to \cB$ be the corresponding morphism.

Then:
\begin{enumerate}
    \item if we denote by $\pi$ the uniformizer of $\qB$, the element $f(\pi)$ is a uniformizer for $\qR$, and
    \item if $\qB\to \qB'$ is an extension of DVRs which induces a finite extension of fraction fields $\operatorname{Frac}(R)\to \operatorname{Frac}(R')$, the tensor product $\qR\otimes_\qB \qB'$ is a DVR.
\end{enumerate}
\end{Lemma}
\begin{proof} 
Point (1) follows as $\qR/f(\pi)\qR\cong \qR\otimes_\qB \qB/\pi$ is the local ring at a generic point of the closed fiber. As the closed fiber is reduced, $\qR/f(\pi)\qR$ is a field, so $f(\pi)$ is a uniformizer.

As for point (2), let $C':= C\times_{\spec(\qB)}\spec(\qB')$, let $\psi:C'\to C$ and $\phi':C'\to \spec(\qB')$ be the two projections. We have an isomorphism $\qR\cong \varinjlim_{\mu \in U} \cO_{C}(U).$
As the tensor product commutes with colimits, $$\qR\otimes_\qB \qB' \cong \varinjlim_{\mu \in U} \cO_{C}(U)\otimes_\qB \qB' \cong \varinjlim_{\mu \in U} \cO_{C'}(\psi^{-1}(U)).$$
The extension of fraction fields $\operatorname{Frac}(R)\to \operatorname{Frac}(R')$ is finite, thus the extension of residue fields is finite; as the residue field of $\qB$ is algebraically closed, this last extension is an isomorphism. We deduce that the closed fibers of $\phi$ and $\phi'$ are isomorphic, and we denote by $\mu'$ the (unique) point over $\mu$ in $C'$. For every open subset $V$ in $C'$ containing $\mu'$, let $D'$ be its complement. Then the irreducible components of $\psi(D')$ are either
locally closed subschemes generically finite over the generic point of $\spec(\qB)$ or closed subschemes of the closed fiber that do not contain $\mu$. Therefore the closure $D:= \overline{\psi(D')}$ does not contain $\mu$ and, if we denote by $U$ its complement, then $\psi^{-1}(U)\subseteq V$. This implies that 
$$\varinjlim_{\mu \in U} \cO_{C'}(\psi^{-1}(U)) \cong \varinjlim_{\mu '  \in V} \cO_{C'}(V) \cong \cO_{C',\mu'}$$
and the latter is a DVR, as $\phi'$ has reduced closed fiber. 
\end{proof}
\begin{Cor}\label{corollary R otimes B K(B) is a field}
With the assumption of \Cref{lemma tensor product thing is a DVR}, if $F$ is the algebraic closure of $K(\qB)$, the fraction field of $\qB$, then $F\otimes_\qB \qR$ is a field. 
\end{Cor}
\begin{proof}
It suffices to check that every non-zero element has an inverse. But a non-zero element of $F\otimes_\qB \qR$ can be written as a non-zero element $\alpha \in K(\qB')\otimes_{K(\qB)}K(\qR)$ for $\qB\to \qB'$ an extension of DVRs: from \Cref{lemma tensor product thing is a DVR}, as $K(\qB')\otimes_{K(\qB)}K(\qR)$ is a field, $\alpha$ has an inverse.
\end{proof}

We now recall two results that we will use later 
\begin{Oss}\label{remark unramified extension if the fiber of residue field is a field}
If $g:\qR\to \widetilde{\qR}$ is an extension of DVRs and $k_\qR$ is the residue field of $\qR$, then if $\widetilde{\qR}\otimes_\qR k_\qR$ is a field, $g$ is unramified.
\end{Oss}
\begin{Teo}[\cite{Galois_cohom}*{Chapter 3, Section 2.4, Theorem 3}]\label{theorem homogeneous space has a point}Assume that $L$ is the field of fractions of a smooth curve over an algebraically closed field, and $\cO\to \spec(L)$ is a homogeneous space for a connected linear group $G$. Then $\cO$ has an $L$-point.
\end{Teo}
In the proof of \Cref{teorema fix} we will use the following lemma, which we isolate from the rest of the proof.
\begin{Lemma}[``A curve through the $k_\qR$-point"]\label{lemma a curve through the pt}
Assume that $\qR$ is a DVR with residue field $k_\qR$, $A$ is a ring, and let $f:\spec(A)\to \spec(\qR)$ be a morphism. Assume that $p\in \spec(A)$ is a closed point that
\begin{enumerate}
    \item is in the closed fiber of $f$, 
    \item is smooth for $f$, and
    \item the inclusion $k_\qR\to k(p)$ is surjective.
\end{enumerate} 
Then there is a DVR $\widetilde{\qR}$ with a morphism $\spec(\widetilde{\qR})\to \spec(A)$ such that the composition $\spec(\widetilde{\qR})\to \spec(A)\to \spec(\qR)$ is surjective, unramified and the extension of residue fields $k_\qR\to k_{\widetilde{\qR}}$ is an isomorphism.
\end{Lemma}
\begin{proof}Let $\pi$ be a uniformizer for $\qR$, and let $\km_p$ be the maximal ideal of $\spec(A)$ corresponding to $p$.
Since $p$ is in the smooth locus of $f$, the image of $\pi$ in the tangent space of $\spec(A)$ at $p$ is not 0. In particular, up to replacing $\spec(A)$ with a neighbourhood of $p$, we can find $n$ elements $f_1,...,f_n$ such that the ideal sheaf corresponding to $p$ is $(f_1,...,f_n,\pi)$, and the images of $f_1,...,f_n,\pi$ in $\km_p/\km^2_p$ are a basis for $\km_p/\km^2_p$. 

Consider then $C$, the closed 
subscheme given by $(f_1,...,f_n)$. Since $f_1,...,f_n$ are linearly independent in $\km_p/\km^2_p$, the closed subscheme $C$ is smooth at $p$ (from the Jacobian criterion) and has dimension one at $p$. In particular, its local ring $\widetilde{\qR}$ at $p$ is a DVR with residue field $k(p)$. As $\pi$ generates the maximal ideal of $p$ in $C$, the map $\spec(\widetilde{\qR})\to \spec(\qR)$ is surjective. It is also unramified from \Cref{remark unramified extension if the fiber of residue field is a field}. 
\end{proof}

\begin{Teo}\label{teorema fix}
Let $R$ be a DVR with algebraically closed residue field, and consider $ C\to \spec(R)$ a flat morphism of finite type, with pure 1-dimensional and smooth fibers. Let $\mu$ be the generic point of an irreducible component of the closed fiber and set $\cB:=\cO_{C,\mu}$.
Assume that:
\begin{enumerate}
    \item[$\operatorname{(i)}$] there is a ring $A$ of finite type over $\qR$ with a $\GL_n$ action,
    \item[$\operatorname{(ii)}$]  there is a universally closed morphism $[\spec(A)/\GL_n]\to \spec(\qR)$, and
    \item[$\operatorname{(iii)}$] there is a section $\phi:\eta\to \spec(A)_\eta$, where $\eta$ is the generic point of $\spec(\qR)$.
\end{enumerate}
Then, up to replacing $\qB$ with a generically finite extension $\qB\to \qB'$, $\qR$ with $\qR'=\qR\otimes_\qB \qB'$ and $\spec(A)$ with $\spec(A\otimes_\qR \qR')$, we can extend the induced section $\widetilde{\phi}:\eta\to [\spec(A)_\eta/\GL_n]$ to a section $\spec(\qR)\to [\spec(A)/\GL_n]$.
\end{Teo}
 We will denote by $k_\qR$ the residue field of $\qR$, and with $K(\qR)$ its fraction field; the proof will essentially proceed as follows. First we use Zariski's main theorem to reduce to the case where the stabilizer of $\eta$ is connected. Using \Cref{theorem homogeneous space has a point}, every orbit over $k_\qR$ has a $k_\qR$-point.
 After some reduction steps which involve a refinement of Abhyankar's lemma (\cite{kollar_modbook}*{Lemma 2.53}) and \Cref{lemma a curve through the pt}, we will produce a DVR $\qR'$ with the same residue field as $\qR$ and with an \'etale morphism $h:\spec(\qR')\to \spec(A)$, which maps the generic point $\eta'=\spec(K(\qR'))$ of $\spec(\qR')$ to a point in the same orbit as $\phi(\eta)$. As $\psi(\eta')$ and $\phi(\eta)$ will be in the same orbit, they will be isomorphic over $\overline{K(\qR)}$, the algebraic closure of $K(\qR)$.
 
Then, using non-abelian cohomology, we will show that after possibly replacing the DVR $\qR$ with the DVR $\qB\otimes_R R'$, for a generically finite extension of DVRs $R\hookrightarrow R'$, the points $\psi(\eta')$ and $\phi(\eta)$ give points in the stack $[\spec(A)/\GL_n]$ which are isomorphic \textit{over} $\spec(K(\qR'))$. Finally we use Rydh's theorem \cite{rydh2011etale}*{Theorem B} to argue that $\spec(\qR)$ is the push-out of $\spec(K(\qR'))\to \spec(K(\qR))$ and $\spec(K(\qR'))\to \spec(\qR')$ so to obtain the desired morphism $\spec(\qR)\to [\spec(A)/\GL_n]$.

\begin{EG} 
To navigate the argument, one can consider the example of $\cB\PGL_{m,\qR}:=[\spec(\qR)/\PGL_m]$, where the action on $\spec(\qR)$ is the trivial one. This stack is isomorphic to a stack of the form $[\spec(A)/\GL_n]$, and our data correspond to a Severi-Brauer variety $X_\eta$ over the generic point $\eta$ of $\spec(\qR)$. 

In this case, the aforementioned morphism $\psi$ could be taken to be given by the trivial Severi-Brauer variety $\spec(\qR)\times\bP^m$ over $\spec(\qR)$, and the base change of $\qB$ is needed in order to trivialize $X_\eta$. 

A similar argument applies also to stacks $\cB G_{\qR}:=[\spec(\qR)/G]$ where $G$ is a connected reductive group: in fact, Steinberg's theorem easily implies that after a suitable extension $\qB\to \qB'$ the pullback of a $G$-torsor over $\eta$ becomes trivial, hence we can extend it in a trivial way.
\end{EG}
We break the proof of the theorem into several lemmas. In what follows, we will denote by $k_\qR$ the residue field of $\qR$, with $K(\qR)$ its fraction field, with $G$ the linear algebraic group $\GL_n$ and with $\ast$ the action of $G$ on $\spec(A)$. 
The following four lemmas will be reduction steps: we will modify $A$ and extend $\qB$ to reduce to the case where $\spec(A)\to \spec(\qR)$ has reduced closed fiber, $\spec(A)$ is normal, the stabilizer of $\phi(\eta)$ is geometrically connected and the orbit of $\phi(\eta)$ is dense. In each reduction step, we will need to check two things. First, we need to check that assumptions (i), (ii) and (iii) still hold, after each modification of $A$. Then we need to check that, after modifying $A$ and $\qB$ in the $n^{\rm th}$ step, the assumptions we achieved in the $(n-1)^{\rm th}$ step still hold.
After the preparatory lemmas, we will use \Cref{theorem homogeneous space has a point} and \cite{rydh2011etale}*{Theorem B} to conclude the proof.
\begin{Lemma}\label{lm:step 1}
    If \Cref{teorema fix} holds when the stabilizer of $\phi(\eta)\in \spec(A)$ is geometrically connected, then it holds in general.
\end{Lemma} 
\begin{proof}
    Since the stabilizer of $\phi(\eta)$ has a $K(\qR)$-point (the identity), it follows from \cite[Exercise 3.2.11]{liubook} that if we can reduce to the case of a connected stabilizer, we can then further reduce ourselves to the geometrically connected case.

The action of $G$ induces a morphism $$\xi: G_\eta\to \spec(A)\text{ that sends }g\mapsto g*\phi(\eta).$$ If $S^\circ$ denotes the connected component of the identity of the stabilizer of $\phi(\eta)$, we can consider the quotient scheme $G_\eta /S^\circ$ and the induced morphism $G_\eta/S^\circ \to \spec(A)$ which is $G_\eta$-equivariant and quasi-finite. Therefore the composition 
$$[(G_\eta/S^\circ)/G_\eta] \longrightarrow [\spec(A)_\eta/G_\eta]\longrightarrow [\spec(A)/G]$$
satisfies the assumptions of Zariski's main Theorem \cite[Théorème 16.5]{LMB}, so it factors as an open embedding and a finite morphism as below:
$$[(G_\eta/S^\circ)/G_\eta] \hookrightarrow \cZ\xrightarrow{\Phi} [\spec(A)/G].$$
We check that the assumptions (i), (ii) and (iii) hold for $\cZ$.

(i) As $\Phi$ is finite, it is affine, so $\cZ\times_{[\spec(A)/G]}\spec(A) = \spec(A')$ and the morphism $\spec(A')\to \cZ$ is a $\GL_n$-torsor: we have $\cZ=[\spec(A')/G]$.

(ii) From how $\xi$ is constructed, the section $\phi$ lifts to $G_\eta$ (via the identity) so to a section $\phi':\eta\to \cZ$, and its automorphism group is $S^\circ$. As $G=\GL_n$ and every $\GL_n$-torsor is trivial over $\eta$, the map $\phi'$ lifts to a map $\phi'':\eta\to \spec(A')$, and its stabilizer is isomorphic to $S^\circ$ (so in particular it is connected).

(iii) As $\cZ\to [\spec(A)/G]$ is universally closed, also $\cZ\to \spec(\qR)$ will be universally closed. So replacing $\spec(A)$ with $\spec(A')$, we can assume that the stabilizer of $\phi(\eta)$ is connected.
\end{proof}
\begin{Lemma}\label{lm:step 2}
 If \Cref{teorema fix} holds when $\spec(A)$ is irreducible and the orbit of $\phi(\eta)$ is dense and geometrically connected, then it holds in general.
\end{Lemma}
\begin{proof}
By \Cref{lm:step 1} we can assume that the stabiliser of $\phi(\eta)$ is geometrically connected.

We replace $\spec(A)$ with the closure of the orbit $\cO(\eta)$ of $\phi(\eta)$, which we denote by $\overline{\cO(\eta)}$, and which is connected as $\cO(\eta)$ is connected since $G=\GL_n$. To check that $\cO(\eta)$ is open in $\overline{\cO(\eta)}$ we use Chevalley's theorem. The morphism $G\to \overline{\cO(\eta)}$, $g\mapsto g*\eta$ is dominant and of finite type, so its image contains a non-empty open subset $U$. Then $\cO(\eta)$ is covered by $gU$ for $g\in G$.

We check that (i), (ii) and (iii) hold. It is clear that (i) and (iii) hold. As for (ii), since the map $[\overline{\cO(\eta)}/G]\to [\spec(A)/G]$ is universally closed, so also its composition $[\overline{\cO(\eta)}/G]\to [\spec(A)/G]\to\spec(\qR)$ is universally closed, and satisfies the assumption of the theorem. It is clear that the stabilizer of $\phi(\eta)$ is still connected.
\end{proof}
\begin{Lemma}\label{lm:step 3}
    If \Cref{teorema fix} holds when $\spec(A)$ is normal, connected and irreducible, and the orbit $\cO(\eta)$ of $\phi(\eta)$ is dense and geometrically connected, then it holds in general.
\end{Lemma}
\begin{proof}
By \Cref{lm:step 2} we can assume that $\spec(A)$ is irreducible and the orbit of $\phi(\eta)$ is dense and geometrically connected.
We replace $\spec(A)$ with the normalization $\nu:\spec(A^n)\to \spec(A)$, since $\nu$ is an isomorphism over $\cO(\eta)$. From the universal property of the normalization $G$ acts on $\spec(A^n)$, so (i) holds. As the normalization is finite, also (ii) holds. Finally, as the normalization is an isomorphism over $\cO(\eta)$, point (iii) still holds. It is clear that the stabilizer of $\phi(\eta)$ is still dense and connected, and $\spec(A)$ is irreducible.
\end{proof}
\begin{Lemma}\label{lm:step 4}
If \Cref{teorema fix} holds when $\spec(A)$ is normal, connected and irreducible, the orbit $\cO(\eta)$ of $\phi(\eta)$ is dense and geometrically connected, and $\spec(A)\to \spec(\qR)$ has reduced central fiber, then it holds in general.
\end{Lemma}
\begin{proof}
By \Cref{lm:step 3} we can assume that $\spec(A)$ is normal, connected and irreducible, and the orbit $\cO(\eta)$ of $\phi(\eta)$ is dense and geometrically connected.
The hypotheses of 
\cite{kollar_modbook}*{Lemma 2.53} apply: we can perform a ramified extension $\qB\to \qB'$ with ramification index divisible enough so that, if $(A')^n$ is the normalization of $A':=A\otimes_\qB \qB'$, the closed fiber of $\spec((A')^n)\to \spec(\qR\otimes_\qB \qB')$ is reduced. We check that (i), (ii) and (iii) still hold.

(i) From the universal property of the fiber product $G$ acts on $\spec(A\otimes_\qB \qB')$, and from the universal property of the normalization the action of $G $ on $\spec(A')$ extends to an action of $G$ on $\spec((A')^n)$.

(iii) As $\qR\otimes_\qB \qB'$ is a DVR from \Cref{lemma tensor product thing is a DVR},
if we denote by $\eta'$ its generic point $\phi$ induces a section $\phi':\eta'\to \spec(A')$, and its orbit (which we denote by $\cO'$) is open (and smooth). From the universal property of the normalization, and since $\cO'$ is smooth, the normalization is an isomorphism over $\cO'$, so $\phi'$ lifts to a map $(\phi')^n: \eta'\to \spec((A')^n)$, and the orbit of $(\phi')^n(\eta')$ is open. 

Observe that we can assume $\spec((A')^n)$ to be normal and connected. This holds since, up to replacing $\spec((A')^n)$ with its connected component $\spec((A')^n_c)$ containing $(\phi')^n(\eta')$ (which is the orbit closure of $(\phi')^n(\eta')$), we can still assume connectedness.

The stabilizer of $\phi'(\eta')$ is still geometrically connected, and since the normalization is an isomorphism on the orbit of $\phi'(\eta')$, also the stabilizer of $(\phi')^n(\eta')$ is geometrically connected.

(ii) Being universally closed is stable under base change so $[\spec(A')/G]\to \spec(\qR\otimes_\qB \qB')$ is universally closed. Since the normalization is a finite morphism, and the inclusion of a connected component is finite, also $\spec((A')^n_c)\to \spec(A')$ is universally closed, so 
$[\spec((A')^n_c)/G]\to [\spec(A')/G]$ is universally closed. Point (ii) follows as the composition of universally closed maps is universally closed.
\end{proof}

This is the end of our reduction steps. We can now assume that $\spec(A)$ is irreducible and has a dense open orbit which is the orbit of $\phi(\eta)$, and the closed fiber of $\spec(A)\to \spec(\qR)$ is reduced. We will prove \Cref{teorema fix} under these assumptions.
\begin{proof}[Proof of \Cref{teorema fix}]
We break the proof into three steps.

\textbf{Step 1.} We construct a DVR $\widetilde{\qR}$ with generic point $\eta_{\widetilde{\qR}}$ and residue field $k_{\widetilde{\qR}}$, and a morphism $\psi:\spec(\widetilde{\qR})\to \spec(A)$ such that:
\begin{enumerate}
    \item the composition $\spec(\widetilde{\qR})\to \spec(\qR)$ gives an extension of DVRs $\qR\to \widetilde{\qR}$ that is \'etale;
    \item the extension of residue fields $k_\qR\to k_{\widetilde{\qR}}$ is an isomorphism;
    \item $\psi(\eta_{\widetilde{\qR}})\in \cO(\eta)$.
\end{enumerate}

Consider $C_\eta$ the complement of $\cO(\eta)$ in $\spec(A\otimes_\qR K(\qR))$, and let $C$ be its closure in $\spec(A)$. It is clear that $C$ is $G$-invariant and does not contain the whole closed fiber of $\spec(A)\to \spec(\qR)$. Therefore there is an orbit $\cO$ of $G$ contained in the closed fiber, which does not intersect $C$, and whose closed points are all smooth points for $\spec(A)\to \spec(\qR)$. Indeed, as $f$ has reduced fibers, the smooth locus of $f$ is open and dense in the closed fiber. Then one can take the orbit of a smooth point away from $C$. This orbit has a $k_\qR$-point $p$ from \Cref{theorem homogeneous space has a point}. \Cref{lemma a curve through the pt} gives a DVR $\widetilde{\qR}$ and a morphism $\psi:\spec(\widetilde{\qR})\to \spec(A)\to \spec(\qR)$ that satisfies (1) and (2). As $p$ does not belong to $C$ and $C$ is closed under specialization, it also satisfies (3).

\textbf{Step 2.} We show that the point $\psi(\eta_{\widetilde{\qR}})$ constructed in Step 1 and $\phi(\eta)$ are in the same $G_{K(\widetilde{\qR})}$-orbit, after possibly replacing $\qR$ with the DVR $\cB\otimes_R R'$, for $R\hookrightarrow R'$ a morphism of DVRs which is generically finite.

As the statement is only over the generic point of $\spec(\qB)$, we first work over the algebraic closure $F_\qB$ of $K(\qB)$. We denote by $F$ (resp. $\widetilde{F}$) the fiber products $F_\qB\otimes_{K(\qB)}K(\qR)$ (resp. $F_\qB\otimes_{K(\qB)} K(\widetilde{\qR})$).
From \Cref{corollary R otimes B K(B) is a field} we know that $F$ is a field, and since $\qR\to \widetilde{\qR}$ is generically finite, we deduce that $\widetilde{F}$ is a finite union of fields, i.e. $\widetilde{F} = \bigcup_{i=1}^kF_i$.
By assumption the extension $K(\qB)\to K(\qR)$ has transcendence degree 1, and since $K(\qR)\to K(\widetilde{\qR})$ is finite, \textit{the fields $F_i$ are }$C_1$.

We denote by $\phi_F$ the pullback of $\phi$ along $\spec(F)\to \spec(\qR)$, then by Step 1 the stabilizer of $\phi_F(\spec(F))$ is geometrically connected; we denote it by $S$. Then for every $i$ we have the following exact sequence of pointed sets: 
$$ G(F_i)\to (G/S)(F_i)\to H^1(F_i, S).$$
But as $S$ is connected, and $F_i$ is $C_1$, from \cite{Galois_cohom}*{Chapter 3, Section 2.3, Theorem 1'}, the set $ H^1(F_i, S)$ is trivial so the map $G(F_i)\to (G/S)(F_i)$ is surjective. In particular, there are $g_i\in G(F_i)$ that send the pullback of $\phi(\eta)$ to the pullback of $\psi(\spec(F_i))$.
Since only finitely many elements of the field extension $K(\qB)\to F_\qB$ will be needed to write each $g_i$, we can replace the field extension $K(\qB)\to F_\qB$ with a finite field extension $K(\qB)\to K(\qB')$.

Then, consider an extension $\qB\to \qB'$ that on the fraction field induces $K(\qB)\to K(\qB')$ as before, let $\qR':=\qR\otimes_\qB \qB'$ with residue field $k_{\qR'}$ and fraction field $K(\qR')$, and let $\widetilde{\qR}':=\widetilde{\qR}\otimes_\qB\qB'$. There is no reason to believe that  $\widetilde{\qR}'$ is a DVR; however, the map $\qR'\to \widetilde{\qR}'$ is still \'etale, the closed fiber of $\spec(\widetilde{\qR}')\to \spec(\qR')$ is a point which we denote by $p$, and the residue field of $\widetilde{\qR}'$ at p is $k_{\qR}$.
So the local ring of $\cO_{\widetilde{\qR}'}$ at $p$, which we denote by $\widetilde{\qR}'_c$, is a DVR which is \'etale over $\qR$, with residue field $k_\qR$ and whose generic point $K(\widetilde{\qR}'_c)$ is in the same $G(K(\widetilde{\qR}'_c))$-orbit as $\phi'(\spec(K(\qR')))$. 
By replacing $\qR$ with $\qR'$ and $\widetilde{\qR}$ with $\widetilde{\qR}'_c$, we get the desired conclusion.

\textbf{Step 3.} Now we finish the proof.
After Step 2, we have an \'etale extension $\qR\to \widetilde{\qR}$ which induces an isomorphism on residue fields, and with a diagram as follows:
$$\xymatrix{\spec(K(\widetilde{\qR}))\ar[d] \ar[r] & \spec(K(\qR)) \ar[d] \ar[r] & [\spec(A)/G] \ar[d] \\ \spec(\widetilde{\qR})\ar[r] \ar[rru] & \spec(\qR) \ar[r]^\Id & \spec(\qR).}$$
As the extension $\qR\to \widetilde{\qR}$ gives an isomorphism of residue fields, from \cite{rydh2011etale}*{Theorem B} the square on the left above is a push-out. Then there is a map $\spec(\qR)\to [\spec(A)/G]$ as desired.
\end{proof}
\section{Extending maps from deformations of curves II}\label{section technical results}
In the previous Section we proved some results about extending maps from the generic fiber of a family of curves over a DVR to the generic points of the closed fiber, i.e. to codimension one points of the family. In this Section we focus instead on extending maps to codimension two points. We will see that in order to do this we might need to modify a bit the closed fiber, i.e. by adding either a stacky structure (twisted nodes) or by adding a whole new component (destabilizing stacky $\cP^1$s).
 We also introduce twisted blow-ups of smooth surfaces (\Cref{sub:twisted blow-ups}).

\subsection{Purity lemma for smooth points}
First, we show how to extend maps when the domain is a pointed surface that satisfies some regularity assumptions.

\begin{Prop}\label{lemma_generalization_purity_tsm1}
    Assume that $S$ is a separated Deligne--Mumford stack of dimension 2, with finite quotient singularities. Let $\cX$ be an algebraic stack with a good moduli space $\cX\to X$, and let $s\in S$ be a closed point of $S$. Assume that there is a diagram as follows:
    \[
    \xymatrix{S\smallsetminus\{s\}\ar[d]\ar[r] & \cX\ar[d]\\S\ar[r]^f & X.}
    \]
    {Then there is a unique $S_2$ Deligne--Mumford stack $\cS$ with $\cS\to S$ a relative coarse moduli space that is an isomorphism on $S\smallsetminus \{s\}$, and such that there is an extension $\phi:\cS\to \cX$ which is representable. Such an extension is unique, the stack $\cS$ has finite quotient singularities, and if $S$ is smooth then $\cS\to S$ is an isomorphism.}  
\end{Prop}
In particular, if $S$ is a scheme then one can construct $\cS$ as above, having $S$ as coarse moduli space.
\begin{proof}From descent and from the uniqueness part of the statement, it suffices to treat the case when $S$ is a scheme. First recall that if $\cS\to S$ is a coarse moduli space morphism from a Deligne--Mumford stack $\cS$, then $\cS\to S$ is separated. Indeed, from \cite{luna}*{Theorem 4.12}  \'etale locally  it is of the form $[\spec(A)/\Gamma]\to \spec(A^\Gamma)$ for a finite group $\Gamma$,  and the latter is a separated morphism.

Since $S$ has finite quotient singularities, from \cite{Vis}*{Proposition 2.8} there is a smooth Deligne--Mumford stack $\cY$ with coarse moduli space $\cY\to S$ that is an isomorphism on the smooth locus of $S$.
There is a map $\cY\smallsetminus\{s\}\to \cX$, and to show that it extends to $\cY\to \cX$ it suffices to show that it extends uniquely up to passing to an \'etale neighbourhood of $s\in \cY$. We choose an \'etale neighbourhood of $s$ by pulling back an \'etale neighbourhood of $f(s)\in X$. From \cite{AHH}*{Proposition 2.10}, there is a cartesian square as follows:
\[
\xymatrix{[\spec(A)/\GL_n]\ar[r]\ar[d] & \cX\ar[d] \\ f(s)\in \spec(A^{\GL_n})\ar[r] & X.}
\]
Now the desired extension $\cY\to [\spec(A)/\GL_n]$ exists and is unique from \Cref{lemma smooth purity}, so from descent there is a unique extension $\cY\to \cX$ of $\cY\smallsetminus \{s\}\to \cX$.

The morphism $\cY\to \cX$ might not be representable, but we can take the relative coarse moduli space $\cY\to \cS\to \cX$ of $\cY\to \cX$. Then $\cS$ has finite quotient singularities, as it is a relative coarse moduli space of a smooth Deligne--Mumford stack, and has a representable morphism $\cS\to 
\cX$. 

We now show that $\cS$ is unique, by contradiction. Let $\cT\to \cX$ be another representable extension, where $\cT$ is an $S_2$ Deligne--Mumford stack with coarse space $S$ and such that $\cT\smallsetminus\{s\}\simeq S\smallsetminus\{s\}$. Then by the same argument as above (except replacing $\cX$ with $\cT$), there is a morphism $\cY\to\cT$ extending the section $S\smallsetminus\{s\}\to \cT\smallsetminus\{s\}$. As $\cY\to S$ and $\cT\to S$ are coarse moduli spaces, $\cY\to \cT$ is a relative coarse moduli space from \Cref{lemma_map_between_S2_DM_stacks_that_is_isom_in_condim_1_with_same_cms_is_relative_cms}. Then $\cT\cong \cS$ as they are both relative coarse moduli spaces of $\cY\to \cX$. 
The map $\phi$ is unique from \Cref{lemma_two_maps_that_agree_in_codim_2_are_the_same}.
\end{proof}
\begin{Lemma}\label{lemma_map_between_S2_DM_stacks_that_is_isom_in_condim_1_with_same_cms_is_relative_cms}
    {Let $\phi:\cX\to \cY$ be a morphism of separated Deligne--Mumford stacks with coarse moduli spaces $X$ and $Y$ respectively, and such that the morphism $X\to Y$ is an isomorphism. Assume that $\cX$ and $\cY$ are $S_2$, and that $\phi$ is an isomorphism over an open dense $V\subseteq \cY$ with complement of codimension at least two. Then the natural map $\cO_\cY\to \phi_*\cO_\cX$ is an isomorphism. In particular, $\phi$ is a relative coarse moduli space.} 
\end{Lemma}
\begin{proof}
{Observe first that $\cX\to X$ is proper: it is universally closed from \cite{AHH}*{Theorem A.8}, of finite type by our conventions, and separated as $\cX$ is separated and $\cY$ has separated (in particular, proper) diagonal.
Since statement is local on $\cY$, we can replace $\cY$ with an \'etale cover of it $U\to \cY$. Let $\cX_U:=U\times_\cY \cX$ and let $\cX_U\to X_U$ its coarse moduli space. Observe that $X_U$ is $S_2$. Indeed, it suffices to check it is such \'etale locally. From \cite{AV_compactifying}*{Lemma 2.2.3} every point $p\in X_U$ admits an \'etale neighbourhood $\spec(A^\Gamma)\to X_U$ such that the following diagram is cartesian, with $\Gamma$ a finite group \[\xymatrix{[\spec(A)/\Gamma]\ar[r] \ar[d] & \cX_U \ar[d] \\ \spec(A^\Gamma)\ar[r] &X_U.}\] Since $\cX$ is $S_2$, also $\cX_U$ and $\spec(A)$ are $S_2$,
so $\spec(A^\Gamma)$ and $X_U$ will be $S_2$ from \cite{KM98}*{Proposition 5.4}. Now, $X_U\to U$ is quasi-finite as $\cX\to \cY$ is such, since $X\to Y$ is an isomorphism. The map $X_U\to U$ is proper as $\cX\to \cY$ is such, so it is finite, hence affine. But $\cO_{X_U}(X_U)$ and $\cO_U(U)$ agree as both $U$ and $X_U$ are $S_2$, the map $\cX_U \times_\cY V\to U\times_\cY V$ is an isomorphism by assumption, and $U\times_\cY V$ has complement of codimension at least two.}
\end{proof}
\begin{Lemma}\label{lemma_two_maps_that_agree_in_codim_2_are_the_same}
    Let $\cX$ be an algebraic stack with affine diagonal, let $S$ be an $S_2$ algebraic stack, and let $U\subset S$ an open subset with complement of codimension at least two. Assume we are given two morphisms $f,g: S\to \cX$ which are isomorphic when restricted on $U$. Then the isomorphism extends uniquely to $S$.
\end{Lemma}
\begin{proof}
Using descent, can assume that $S=\spec(A)$ is an affine scheme.
    Consider $S\to \cX\times \cX$ given by $(f,g)$, and let $I:=\cX\times_{\cX\times \cX}S\to S$ be the second projection. Sine the diagonal of $\cX$ is affine, $I\to S$ is affine, so $I=\spec(B)$. There is an open $U\subseteq S$ where $f$ and $g$ are isomorphic, and thus a map $U\to\spec(B)$ inducing $B\to \cO_S(U)$. As $S$ is $S_2$, we have that $\cO_S(U)=\cO_S(S)$ and as $S$ is affine $\cO_S(S) = A$, so we have a map $B\to A$. This induces the desired map $S\to I$ extending $U\to I$.  
\end{proof}

\begin{Lemma}\label{lemma smooth purity}
Consider the following diagram, where $X$ is an $S_2$ Deligne-Mumford stack, $p$ a closed point and $G$ is a linear algebraic group:
$$\xymatrix{X\smallsetminus \{p\}\ar[d] \ar[r]^-a &[\spec(A)/G]\ar[d] \\ X \ar[r] & \spec(A^{G})}$$
Assume that $G=\GL_n$ or $G=\Gm^n$, and that either one of the following holds:
\begin{enumerate}
    \item $X$ is smooth of dimension 2, or
    \item The map $X\smallsetminus \{p\}\to [\spec(A)/G]\to \cB G$
 extends to $X\to \cB G$ and $\dim(X)\ge 2$.\end{enumerate}
Then there is a unique lifting $X\to [\spec(A)/G]$.
\end{Lemma}
\begin{proof}We do the case $G=\GL_n$, the case $G=\Gm^n$ is analogous.
As the lifting will be unique, up to working on an \'etale neighbourhood of $p$ we can assume that $X$ is affine.
We first show that if $X$ has dimension 2, then the map $X\smallsetminus \{p\}\to [\spec(A)/\GL_n]\to \cB\GL_n$
 extends to $X\to \cB\GL_n$. Indeed, the morphism $a$ induces a $\GL_n$-torsor $F\to X\smallsetminus \{p\}$, together with a $\GL_n$-equivariant map $F\to \spec(R)$. Let $\cF$ be the rank $n$ vector bundle on $X\smallsetminus \{p\}$ induced by $F$. If $i:X\smallsetminus \{p\}\hookrightarrow X$ is the inclusion, then $\cG' = i_*\cF$ is a reflexive sheaf on a smooth surface. From \cite{Har80}*{Corollary 1.4} $\cG'$ is a vector bundle. In particular, from the equivalence between vector bundles of rank $n$ and $\GL_n$-torsors, there is a $\GL_n$-torsor $\cG\to X$ extending $F$.

Observe now that:
\begin{enumerate}
    \item the map $\cG\to X$ is affine so $\cG$ is affine,
    \item $\cG$ is $S_2$, and
    \item $F$ is open in $\cG$, with complement of codimension at least 2.
\end{enumerate}
Our goal is to extend the map $\alpha: F\to \spec(A)$ to a $G$-equivariant map $\cG\to \spec(A)$. But $\alpha$ induces $A\to H^0(\cO_F)$ and from points (2) and (3) above, $H^0(\cO_F) = H^0(\cO_\cG)$. In particular, we have a map $A\to H^0(\cO_\cG)$ which is compatible with the $G$-action. This induces the desired equivariant map $\cG\to \spec(A)$.
\end{proof}

\subsection{Twisted blow-up}\label{sub:twisted blow-ups}
In what follows, we define \emph{twisted blow-ups} of smooth surfaces: these are birational transformations in which a point is replaced by a stacky $\cP^1$. This construction will be needed for the proof of \Cref{proposition val criterion for the nodes}.
\begin{Lemma}\label{lemma blow up + covering stack does the job}
Let $R$ be a DVR with parameter $\pi$ and let $\cA$ be a smooth 2-dimensional local $R$-algebra with maximal ideal $\mathfrak{m}=(\pi,y)$, let $d> 0, m> 0$ two positive integers. Let $S:=\spec(\cA/(y))$ be the closed subscheme of $\spec(\cA)$ of equation $y=0$ and let $B\to \spec(R)$ be the blow-up of $\spec(\cA)$ along the ideal $(\pi^{md},y)$.

Then there is a Deligne--Mumford stack $\cB_{m,d}$ with coarse space $\rho:\cB_{m,d}\to B$ and a line bundle $\cI$ on $\cB_{m,d}$ such that:
\begin{enumerate}
    \item the closed substack of $\cB_{m,d}$ given by $\pi = 0$ is a nodal twisted curve;
    \item the closed embedding $S\to \spec(\cA)$ lifts to a closed embedding $j:S\to \cB_{m,d}$;
    \item we have $j^*\cI = \cO_S(-mp)$ where $p$ is the closed point on $S$;
    
             \item the exceptional divisor of $b:\cB_{m,d}\to \spec(\cA)$ is isomorphic to a root stack of order $d$ of $\bP^1$ at a point;
    \item the degree of $\cI$ restricted to the exceptional divisor of $b$ is $\frac{1}{d}$;
    \item  $b_*\cI^{\otimes k} = \cO_{\spec(\cA)}$ for every $k\le 0$, and
        $(\pi^{mk},y^{\lceil \frac{k}{d}\rceil})\subseteq b_*\cI^{\otimes k}$ for every $k>0$.
\end{enumerate}
\end{Lemma}
\begin{Def}
The stacks $\cB_{m,d}$ will be called an $(m,d)$-blow up, and $\cI$ will be called the ideal sheaf of the $(m,d)$-exceptional locus. When the integers $(m,d)$ are unnecessary, we will just call it a \emph{twisted blow up}. 
\end{Def}
The situation is summarized in the following diagram:
$$\xymatrix{\cB_{m,d}\ar[r] & B\ar[r] & \spec(\cA) \\ & & S\ar[ull]^j\ar[u]}$$
Point (6) in \Cref{lemma blow up + covering stack does the job} will allow us to control the algebra $\bigoplus_i b_*\cE^i$.
\begin{Remark}
 The subscript $m$ records the index in condition~(3), and $d$ the order of the
stabilizer at the non-schematic point of $\cB_{m,d}$. This construction plays
a key role in \Cref{proposition val criterion for the nodes}; we refer to
\Cref{remark_def_B_md} for further motivation.
\end{Remark}
\begin{proof}[Proof of \Cref{lemma blow up + covering stack does the job}]The idea of the construction is as follows. We blow up $\mathrm{Spec}(A)$ at the weighted ideal $(\pi^{md}, y)$ to produce a surface $B$ in which the proper transform of $S$ no longer meets the proper transform of $\{\pi=0\}$. After, we extract a $d$-th root of the exceptional divisor to produce the line bundle $\mathcal{I}$ with the required degree and gluing properties.

\medskip

Recall that $B=\operatorname{Proj}_{\cA}(\cA[u,v]/(\pi^{md}u-yv))$. The chart $D(u)$ where $u\neq 0$ is isomorphic to
$$\spec(\cA\left[\frac{v}{u}\right]/(\pi^{md}-y\frac{v}{u}))$$ and the exceptional divisor has equation $y=0$.
The chart $D(v)$ where $v \neq 0$ instead is isomorphic to $$\spec(\cA\left[\frac{u}{v}\right]/(\frac{u}{v}\pi^{md}-y))\cong \spec(R\left[\frac{u}{v}\right]),$$ the exceptional divisor is generated by the ideal $\pi^{md}$, and the proper transform of $S$ is given by $\frac{u}{v}=0$. Observe also that $D(u)$ is smooth away from a single point (which we denote by $q$), and at that point there is a cyclic quotient singularity. We denote by $\cB'\to B$ the covering stack that resolves this singularity.

Now consider the $\bQ$-Cartier divisor
given by the $d$-th root of the exceptional divisor on $
B\smallsetminus\{q\}$. This is the divisor with equation $\pi^m=0$ on $D(v)$ and 1 away from the exceptional divisor. It induces a map $B\smallsetminus \{q\} \to \cB\Gm$ that from \Cref{lemma smooth purity} gives a map $\cB'\to \cB\Gm$. Let $\cB_{m,d}$ be the relative coarse moduli space of this map, and let $\cI$ be the dual of the resulting line bundle on $\cB_{m,d}$. In particular, $\cI$ is the ideal sheaf generated by $\pi^m$ on $D(v)$ (where $\cB_{m,d}\to B$ is an isomorphism) and 1 on the locus where $\cB_{m,d}\to B$ is an isomorphism. We check the desired conditions for $\cB_{m,d}$. As now $m$ and $d$ play no role in the rest of the argument, we drop them (namely, we define $\cB:=\cB_{m,d}$).

To check (1), we perform a local analysis of $\cB$ on the stacky point, and we denote by $v':=\frac{v}{u}$. Smooth-locally, the covering stack $\cB'$ on the stacky points is isomorphic to $[\spec(k[\![w,z]\!])/\bmu_{md}]$ with the action given my $\xi*w = \xi w$ and $\xi* z = \xi^{-1}w$ and the map $\spec(k[\![w,z]\!])\to \spec(\cA[v']/\pi^{md} - yv')$ given by $\pi\mapsto wz$ $y\mapsto z^{md}$ and $v'\mapsto w^{md}$. Our stack $\cB$ is obtained by first taking the coarse space of $[\spec(k[\![w,z]\!])/\bmu_{m}]$, which has an action of $\bmu_d \cong \bmu_{md}/\bmu_m$, and then taking the stack quotient of this coarse space by $\bmu_d$. In particular:
\begin{itemize}
    \item the coarse space of $[\spec(k[\![w,z]\!])/\bmu_{m}]$
is isomorphic to $\spec(k[\![\alpha,\beta,\gamma]\!]/\alpha\beta - \gamma^m)$, 
\item the action of $\bmu_d$ leaves $\gamma$ invariant, and a generator $\xi$ of $\bmu_d$ sends $\alpha\mapsto \xi\alpha$ and $\beta \mapsto \xi^{-1} \beta$, and
\item the quotient map $\spec(k[\![\alpha,\beta,\gamma]\!]/\alpha\beta - \gamma^m)\to \spec(R[v']/\pi^{md} - yv') $ sends $\pi\mapsto \gamma$, $y\mapsto \alpha^d$ and $u\mapsto \beta^d$.
\end{itemize}   Therefore the ideal $(\pi)$ gets sent to the ideal of $(\gamma)$, and $\spec(k[\![\alpha,\beta,\gamma]\!]/(\alpha\beta - \gamma^m, \gamma))\cong \spec(k[\![\alpha,\beta]\!]/\alpha\beta)$ is the desired twisted nodal curve.

Points (2) and (3) follow from the construction of $\cB$. Point (4) follows from the explicit description of $\cB$ (or from \cite{GS17}, since the exceptional divisor with the reduced structure is smooth with coarse space $\bP^1$). Point (5) follows from a local analysis: $\cI^{d}$ is the ideal sheaf of the exceptional divisor, and it has corresponding Weil divisor $-(md)E$ where $E$ is the reduced exceptional divisor of $B\to \spec(\cA)$. From the singularities of $B$ one can check that $E^2=-\frac{1}{md}$, so the degree of $\cI^{d}$ restricted to the exceptional is $1$. Then the degree of $\cI$ is $\frac{1}{d}$.
The first part of (6) follows since we have a map $\cO_{\cB}\to \cI^{\otimes k}$ as
$\cI$ is an ideal sheaf, for $k\le 0$. We can push forward this map to get $\phi:\cO_X\to b_*\cI^{\otimes k}$. The sheaf $b_*\cI^{\otimes k}$ agrees with
$\cO_{\spec(\cA)}$ away from the closed point, which has codimension 2, so $(b_*\cI^{\otimes k})^{**}\cong \cO_{\spec(\cA)}$. One can now check that $\psi$ and $\phi$ are inverse to each other. For the second part of (6), so for $k>0$, it suffices to check that for the two sections $\pi^{mk}$ and $y^{\lceil \frac{k}{d}\rceil}$ vanish on the ideal sheaf $\cI^{k}$. Again this can be checked locally on $\cB\smallsetminus \{q\} = B\smallsetminus \{q\} = D(v)$, as $\cI$ is $S_2$. Here the desired statement follows from the explicit description of $D(v)$. 
\end{proof}
\begin{Lemma}\label{lemma blow up moreover part}
    If there is a cyclic group $\bmu_\ell$ that acts on $\cA$ by fixing $\pi$ ($\xi*\pi = \pi$) and $\xi*y = \xi y$, this action can be extended to $\cB_{m,d}$ in such a way that the action is faithful on the exceptional divisor of $b$, fixing only the stacky point and the intersection of the exceptional divisor with $j(S)$, and the quotient $[\cB_{m,d}/\bmu_\ell]$ has cyclic stabilizers.
\end{Lemma}
\begin{proof}
Observe that if we can put the same action of $\bmu_\ell$ on $y$ also on $u$ (namely, $\xi*u = \xi u$) then this action can be extended to $B$ and therefore also to $\cB'$. As the divisor $\pi^m$ is $\bmu_\ell$-invariant, the action extends also on $\cB$. We now check that the stabilizers are cyclic, by explicitly describing this action on $\cB$. As $\bmu_\ell$ is cyclic and $\cB$ is smooth away from its stacky point, we just check how the action is defined on the stacky point of $\cB$. Recall that smooth locally at this stacky point, $\cB$ is isomorphic to $[\spec(k[\![\alpha,\beta,\gamma]\!]/(\alpha\beta - \gamma^m))/\bmu_d]$, and the action of $\bmu_\ell$
on its coarse space is trivial on $\pi$ and the characters of $y$ and $v'= \frac{v}{u}$ are inverse to each other. We can consider then the action of $\bmu_{\ell d}$ on $\spec(k[\![\alpha,\beta,\gamma]\!])$ that leaves $\gamma$ invariant, and acts faithfully with dual characters on $\beta$ and $\alpha$. One can check that this gives an action of $\bmu_\ell = \bmu_{\ell d}/\bmu_d$ on $[\spec(k[\![\alpha,\beta,\gamma]\!]/(\alpha\beta - \gamma^m))/\bmu_d]$ which extends the desired action: the stabilizers are cyclic as they are a subgroup of $\bmu_{\ell d}$. The commutativity with \'etale base changes follows from the explicit construction of $\cB$.
\end{proof}
Finally, the following is a straightforward consequence of how we constructed $\cB$.
\begin{Lemma}\label{lemma blow up last part}
    The constructions of $\cB_{m,d}$ and $\cI$ commute with \'etale base change, in the following sense: let $\spec(\cA')\to \spec(\cA)$ be an \'etale morphism, and assume that the fiber of $\mathfrak{m}$ is a unique maximal ideal $\mathfrak{m}'$ generated by the pullback of $\pi$ and $y$. Let $\cB_{m,d}'$ and $\cI'$ be the stacks obtained by applying the first part of this lemma to $\mathfrak{m}'$. Then $\cB_{m,d}'\cong\cB_{m,d}\times_{\spec(\cA)} \spec(\cA')$, and $\cI'$ is the pull-back of $\cI$.
\end{Lemma}
\subsection{Lifting line bundles via twisted blow-ups}\label{subsection_md_blowups}
We now study the case of a nodal singularity that does not smooth over a DVR $R$ with uniformizer $\pi$. 
This will be the pushout of a diagram that has the two branches of the node $B_1$ and $B_2$, together with the inclusions of the singular locus $S\to B_1$ and $S\to B_2$. We adopt the following.

\begin{Notation}\label{notation node}
 We denote by $R_1$ and $R_2$ the localizations of two smooth $R$-algebras of dimension 2, so that $B_i=\spec(R_i)$, and the maps $R\to R_i$ are injective. Let $\mathfrak{m}_1=(\pi,x)$ (resp. $\mathfrak{m}_2=(\pi,y)$) be maximal ideal of $R_1$ (resp. $R_2$). As $R_1/(x)\cong R_2/(y)$ are isomorphic to $R$, we identify both with $R:= R_1/x$.
So $S=\spec(R)$ and let $i_1:S\to B_1$ (resp. $i_2:S\to B_2$) be the closed embedding given by $x=0$ (resp. $y=0$). 
\end{Notation}

 We can then form a cocartesian diagram
\begin{equation}\label{eq_pushout}
\begin{tikzcd}
S \ar[r, hookrightarrow] \ar[d, hookrightarrow] & B_1 \ar[d, hookrightarrow] \\
B_2 \ar[r, hookrightarrow] & N
\end{tikzcd}
\end{equation}
The scheme $N$ has a (constant) nodal singularity over $\spec(R)$, where $B_i $ are the two branches of the node. We denote $p\in N$ the single node over the closed point of $\spec(R)$.
\begin{Oss}\label{remark_def_B_md}\textcolor{black}{
We can now elaborate further on the definition of $\cB_{m,d}$ and its purpose. Let $\rho\colon \cB_{m,d}\to \spec(\cA)$
be as in \Cref{lemma blow up + covering stack does the job}, with the closed embedding
$S\hookrightarrow \spec(\cA)$ from the same lemma. Within the setup of \Cref{eq_pushout} we think of $\spec(\cA)$ as a local picture of one of the two branches at the singular point of the closed fiber over $\spec(R)$, and $S$ corresponds to the section along which we glue the branches.}

\textcolor{black}{Suppose we are given a map $\phi\colon \spec(\cA)\to [\bA^1/\Gm]$, where $\Gm$ acts with
weight $d>0$. Such a map corresponds to a line bundle $\cL$ on $\spec(\cA)$ together with a
section $\sigma_{\phi}$ of $\cL^{\otimes d}$. The construction of $\cB_{m,d}$ is motivated by
the following problem: we wish to perform a (possibly weighted) blow-up $B\to \spec(\cA)$
centered at the closed point, and to construct a map $\phi'\colon B\to [\bA^1/\Gm]$
satisfying:
\begin{enumerate}
    \item the section $\sigma_{\phi'}$ corresponding to $\phi'$ does not vanish along the
          proper transform of $S$, and
    \item the restriction of $\phi'$ to the complement of the exceptional divisor agrees
          with $\phi$.
\end{enumerate}
Condition~(1) will be used in the third paragraph of \Cref{proposition val criterion for the nodes} to ensure that $\phi'$ can be glued along $S$ to the map defined
on the other branch of the nodal singularity. Condition~(2) ensures that the blow-up does not alter $\phi$ away from the center.}

\textcolor{black}{Now suppose that $\sigma_\phi|_S$ has a zero of multiplicity $m>0$ at the closed point of $S$.
Then there exists a weighted blow-up of $\spec(\cA)$, depending on $m$, which we denote
$B_m\to \spec(\cA)$, such that the proper transform of $\{\sigma = 0\}$, which we denote
$\Sigma$, does not meet the proper transform of $S$. Such a weighted blow-up is not unique:
there are in general multiple weighted blow-ups satisfying this property. However, not every
such blow-up guarantees that the line bundle $\cO_{B_m}(\Sigma)$ is the $d$-th power of
another line bundle, nor that the fiber over $\pi = 0$ admits a simple description.}

\textcolor{black}{The reason we choose $\cB_{m,d}$ is that it satisfies both of these additional requirements:
the line bundle $\cO_{\cB_{m,d}}(\Sigma)$ is the $d$-th power of a line bundle, and the
morphism $\cB_{m,d}\to \spec(\cA)$ can be described explicitly. The subscript $m$ records
the order of vanishing of $\sigma_\phi$ along $S$, while $d$ records the order of the stabilizer
at the stacky point of $\cB_{m,d}$.}
\end{Oss}
With the notation explained in \ref{notation node}, we can now state the following.
\begin{Prop}\label{proposition val criterion for the nodes}
 Let $A$ be a ring and let $\Gm^{n}$ act on $\spec(A)$. Consider the action of the $k$-th $\Gm$ on $\spec(A)$ and the induced $\bZ$-grading on $A$, and let $d\in\mathbb{N}$ be an integer such that the graded ring $A$ is generated in degrees $[-d,d]$. Assume that there is a diagram as follows:
$$\xymatrix{N \smallsetminus \{p\} \ar[r]^-b \ar[d] & [\spec(A)/\Gm^n]\ar[d]\\ N \ar[r]^-{a}  & \spec(A^{\Gm^n}).}$$
By composing with the map $[\spec(A)/\Gm^n]\to \cB\Gm^n$, we have $n$ line bundles on $N\smallsetminus \{p\}$, and assume that the line bundles corresponding to the first $k-1$ factors of $\cB\Gm^n$ extend to $N$, but the remaining $n-k+1$ do not.

Then we can perform a birational transformation $\rho:\cN_d\to N$ which:
\begin{itemize}
    \item is an isomorphism on $\rho^{-1}(N\smallsetminus \{p\})$;
    \item $\cN_d\to \spec(R)$ is flat and the closed fiber is a nodal twisted curve with 3 irreducible components and a single stacky point $q\in \cN_d$;
    \item the map $\rho$ contracts a single (stacky) curve $\cP^1$ whose coarse space is $\bP^1$;
    \item there is a commutative diagram as follows, where $\alpha$ is representable and we denote again by $p$ the intersection of the codimension one nodal locus of $\cN_d$ and the closed fiber:
    $$\xymatrix{\cN_d\smallsetminus \{p\}\ar[r]^-\alpha \ar[d] & [\spec(A)/\Gm^n]\ar[d]\\ N \ar[r]^-a  & \spec(A^{\Gm^n});}$$
    \item the line bundle corresponding to the $k$-th factor of $\cB\Gm^n$ lifts to a line bundle  on $\cN_d$ with degree $\frac{1}{d}$ on the stacky $\cP^1$, and the line bundles corresponding to the first $k-1$ factors of $\cB\Gm^n$ still lift as well.
\end{itemize} 
The construction of $\cN_d$ commutes with \'etale base change.

Moreover, if there is an action of $\bmu_\ell$ on $R_1$ (resp. $R_2$) which is trivial on $\pi$ (i.e. $\xi*\pi = \pi$) and acts faithfully on $x$ (resp. $y$), such that $a$ and $b$ commute with this action, then we can put an action of $\bmu_\ell$ on $\cN_d$ such that $\alpha$ and $\beta$ are equivariant. This action leaves $\cP^1$ invariant and it acts faithfully on it.
\end{Prop}

As already anticipated in \Cref{remark_def_B_md}, the stack $\cN_d$ in the statement of \Cref{proposition val criterion for the nodes} is obtained as follows: first, one performs an $(m,d)$-blow up on one of the two branches $B_i$ (for example, $\cB_{m,d}\to B_2$) centered at the point $p$ of the irreducible component $B_2$. Then, one glues $B_1$ and $\cB_{m,d}$ along the proper transform of $S$.
\begin{center}
\includegraphics[scale=0.075]{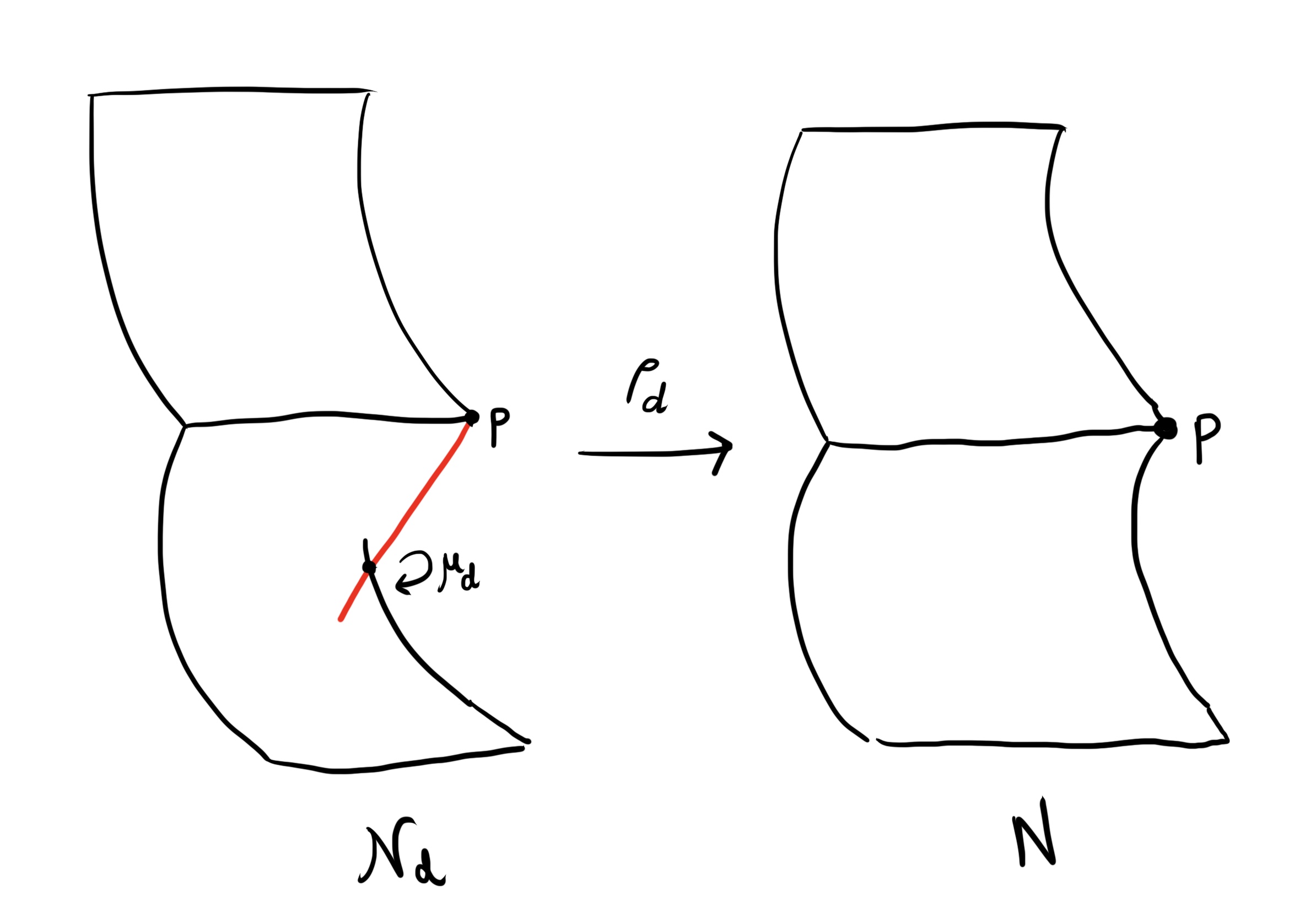} \\ Figure: the stacky surface $\cN_d$.
\end{center}
The integer $d$ depends on the action of the $k^{\rm th}$ factor of $\Gm^n$ on $\spec(A)$, as specified in the statement of \Cref{proposition val criterion for the nodes}. The integer $m$ depends on the map $N\smallsetminus\{p\}\to \cB\Gm^n\to \cB\Gm$, where the last morphism is the projection on the $k^{\rm th}$ factor.

More precisely, the composition above corresponds to a $\Gm$-torsor $T_k\to N\smallsetminus\{p\}$, obtained by gluing the two restrictions ${T_k}_{|B_1\smallsetminus\{p\}}\to B_1$ and ${T_k}_{|B_2\smallsetminus\{p\}}$ along $S\smallsetminus\{p\}=\spec(R_0)$. Such a gluing corresponds to an invertible element $f\in R_0$, and the integer $m$ is obtained by the valuation of $f$ at the closed point of $R$.

\begin{Oss}\label{rmk:extension}
If we can lift \textit{all} the line bundles coming from $[\spec(A)/\Gm^n]\to \cB\Gm^n$ to $N$, then from Proposition \ref{lemma smooth purity} we can uniquely lift the map $a$ to $N\to [\spec(A)/\Gm^n]$.
\end{Oss}

\begin{proof}[Proof of \Cref{proposition val criterion for the nodes}]
First observe that if we can find a map $\alpha$ which satisfies all the conditions except for its representability, then we can find a representable one, as we can take its relative coarse moduli space. 
As our construction of $\cN_d$ will commute with \'etale base change, we consider the case $B_1:=\spec(R[\![x]\!])$ and $B_2 :=\spec(R[\![y]\!])$.
We also denote by $\cX:= [\spec(A)/\Gm^n]$.

We begin by analyzing the first diagram of the proposition. As $N\smallsetminus \{p\}$ is a push-out, the data of a map $N\smallsetminus \{p\}\to \cX$ is equivalent to two maps $\phi_i: B_i\smallsetminus \{p\}\to \cX$, and an isomorphism $\phi_1\circ {i_1}_{,\eta} \to \phi_2\circ {i_2}_{,\eta}$, where ${i_j}_{,\eta}$ denotes the restriction of the inclusion $i_j$ to the generic point $\eta$ of $S$. 

From \Cref{lemma smooth purity}, the two maps $\phi_i$ extend (we still denote the extensions by $\phi_i$), so in order for $a$ not to lift we need that the isomorphism $\phi_1\circ {i_1}_{,\eta} \to \phi_2\circ {i_2}_{,\eta}$ does not extend. 
More explicitly, the maps $\phi_i$ induce $\Gm$-torsors $\{T_{i,j}\to B_i\}_{j=1}^n$, which are trivial torsors as $B_i$ is regular and local. When restricted to $S$, these give torsors $(T_{i,j})_{|S}$,
and we have an isomorphism $f_{T,j}:(T_{1,j})_{|S_\eta}\to (T_{2,j})_{|S_\eta}$ for every $j$ when we restrict these torsors to $\eta$, the generic point of $S$. If we denote by $\cL_i$ the line bundle associated to $T_{i,k}$, the map $f_{T,k}$ induces a map $f_\eta:(\cL_1)_{|S_\eta}\to (\cL_2)_{|S_\eta}$. Up to swapping
$f_{T,k}$ with its inverse, we can assume that this extends to a map $f:(\cL_1)_{|S}\to (\cL_2)_{|S}$ which is an isomorphism on the generic point. More explicitly, the map $f_\eta$ induces an isomorphism $(\cL_1)_{|S}\to (\cL_2)_{|S}(-mp)$ where $p$ is the closed point of $S$, for a certain \textit{strictly positive} integer $m$ (positive as the $k$-th line bundle does not extend to $N$).

Now, the map $\phi_2$ induces a map $\varphi: A\to \cO_{T_2}(\cO_{T_2}) = R[\![y]\!][t_1^{\pm1},...,t^{\pm1}_n]$ where $T_2$ is the $\Gm^n$-torsor over $B_2$, and by definition of $d$ the generators of $A$ as a $k$-algebra map into $\cR t^{-d}_k\oplus \cR t^{-d+1}_k\oplus ... \oplus \cR t^d_k$, where $\cR = R[\![y]\!][t_1^{\pm1},...,t_{k-1}^{\pm1},t_{k+1}^{\pm 1},...,t_n^{\pm 1}]$. For these choices of $d$ and $m$, let $b:\cB_{m,d}\to B_2$ be the blow-up of \Cref{lemma blow up + covering stack does the job}, with line bundle $\cI$,
and let $j:S\to \cB_{m,d}$ be the lifting of $S\to B_2$ to $\cB_{m,d}\to B_2$ of \Cref{lemma blow up + covering stack does the job}. Then $\cI_{|S} = (\cL_2)_{|S}(-mp)$, and consider
$\spec_{\cB_{m,d}}(\bigoplus_{j\in \mathbb{Z}}\cI^j)$ be the $\Gm$-torsor associated to $\cI$. From \Cref{lemma blow up + covering stack does the job}, we have an
inclusion $ b_*(\bigoplus_{j\in \mathbb{Z}}\cI^j) = \bigoplus_{j\in \mathbb{Z}}b_*\cI^j\subseteq R[\![y]\!][t^{\pm i}]$ and since $R[\![y]\!][t_1^{\pm1},...,t_{k-1}^{\pm1},t_{k+1}^{\pm 1},...,t_n^{\pm 1}]$ is a flat
$R[\![y]\!]$-algebra, an inclusion $$ \bigoplus_{j\in \mathbb{Z}}b_*\cI^j\otimes _{R[\![y]\!]}R[\![y]\!][t_1^{\pm1},...,t_{k-1}^{\pm1},t_{k+1}^{\pm 1},...,t_n^{\pm 1}] \subseteq R[\![y]\!][t_1^{\pm1},...,t_n^{\pm 1}]. $$
We interpret both terms in this equality. The right hand side is $\cO_{T_2}(T_2)$: the regular functions on the total space of the $\Gm^n$-torsor over $B_2$.
The left hand side are the regular functions on the total space of the $\Gm^n$-torsor over $\cB_{m,d}$ that on the $j$-th component, for $j\neq k$, is the pull back of $T_{2,j}$, and on the $k$-th component is the torsor associated to the line bundle $\cI$ of Lemma \ref{lemma blow up + covering stack does the job}. We denote this torsor by $\cT$.

Then the situation is as in the following diagram:
$$\xymatrix{ & \cO_\cT(\cT)=\bigoplus_{j\in \mathbb{Z}}b_*\cI^j\otimes _{R[\![y]\!]}R[\![y]\!][t_1^{\pm1},...,t_{k-1}^{\pm1},t_{k+1}^{\pm 1},...,t_n^{\pm 1}]  \ar@{^{(}->}[d] \\ A\ar[r]^-{\varphi} & R[\![y]\!][t_1^{\pm1},...,t_n^{\pm 1}]. }$$
\textit{Claim. } The map $\varphi:A\to R[\![y]\!][t_1^{\pm1},...,t_n^{\pm 1}]$ lifts to $A\to \cO_\cT(\cT)$.

Observe that this would finish the argument. Indeed, we would have a graded map $b^*A\to \cO_\cT(\cT)$, so an equivariant map $\cT \to \spec(A)$. This in turns induces a map $\cB_{m,d}\to \cX$ such that when restricted to $S$ (which embeds in $\cB_{m,d}$ via $j$) gives a $\Gm^n$-torsor whose $k$-th line bundle is $(\cL_2)_{|S}(-mp)$, and such that the isomorphism $f_{T,k}$ extends.

\emph{Proof of the claim.}
Let $g$ be a generator of $A$ in degree $k$. Observe the following commutative diagram, where $K(R)$ is the fraction field of $R$, $\pi$ is the uniformizer of $R$ and $u$ a unit of $R$:
$$\xymatrix{ & \cO_{T_2}(T_2) = R[\![y]\!][t_1^{\pm1},...,t_n^{\pm 1}] \ar@{^{(}->}[dr]\ar[rr]^-{y = 0} & & R[t_1^{\pm1},...,t_n^{\pm 1}] \ar@{^{(}->}[dr]& \\
& & K(R)[\![y]\!][t_1^{\pm1},...,t_n^{\pm 1}]\ar[rr] & & K(R)[t_1^{\pm1},...,t_n^{\pm 1}]\\
A\ar[ruu]^-\varphi \ar[rdd] & & & &\\
& & K(R)[\![x]\!][t_1^{\pm1},...,t_n^{\pm 1}]\ar[rr] & & K(R)[t_1^{\pm1},...,t_n^{\pm 1}]\ar[uu]_{t_k\mapsto \pi^m u t_k}\\
& \cO_{T_1}(T_1) = R[\![x]\!][t_1^{\pm1},...,t_n^{\pm 1}] \ar@{^{(}->}[ur]\ar[rr]^-{x = 0} & & R[t_1^{\pm1},...,t_n^{\pm 1}]. \ar@{^{(}->}[ur]& 
}$$
If we chase the generator $g$ via the diagram above, we see that $\varphi(g)\in (\pi^{mk},y)$. In particular, from \Cref{lemma blow up + covering stack does the job} points (6), every generator of $A$ maps via $\varphi$ to $\bigoplus_{j\in \mathbb{Z}}b_*\cI^j[\![y]\!][t_1^{\pm1},...,t_{k-1}^{\pm1},t_{k+1}^{\pm 1},...,t_n^{\pm 1}]$. Therefore the map $\varphi$ lifts.

The moreover part follows as above, from \Cref{lemma blow up moreover part}.
\end{proof}
\subsection{Iterated twisted blow-ups}\label{subsection_iterated_twisted_blowup} 
 We end this section with a mild generalization of
\Cref{proposition val criterion for the nodes}, which can be understood as
follows. In the proof of \Cref{main theorem intro} we will use
\Cref{proposition val criterion for the nodes} iteratively, to lift each
$\Gm$-torsor associated to
$N\smallsetminus\{p\}\to[\spec(A)/\Gm^n]\to\cB\Gm^n$; once all the
$\Gm$-torsors have been lifted, \Cref{rmk:extension} allows us to lift the
map itself. Since each application of
\Cref{proposition val criterion for the nodes} requires passing to an
\'etale neighbourhood of a node, iterating the argument would involve
repeatedly introducing new \'etale neighbourhoods. Doing so presents no
essentially new difficulty, but we prefer to isolate these technicalities
from the proof of \Cref{main theorem intro} in order to simplify its
exposition. This is why we introduce
\Cref{cor:val criterion for the nodes iterated}.

 With the same notation as in \Cref{notation node}, let $B_2$ be one of the two branches and let $\cB_2^{(1)}\to B_2$ be a twisted blow up of $B_2$ at $p$: observe that in $\cB_2^{(1)}$ we have a distinguished point that lies in the intersection of the proper transform of $S$ with the exceptional divisor, and that the stabilizer of this point is always trivial. With a little abuse of notation, we call such a point $p$ and we denote the proper transform of $S$ in $\cB_2^{(1)}$ by the same name.

By construction $\cB_2^{(1)}$ is a Deligne-Mumford stack and
\begin{enumerate}
    \item there is a Zariski open morphism $\cB_2^{(1)}\smallsetminus\{p\}\hookrightarrow \cB_2^{(1)}$;
    \item we can take an \'etale neighbourhood $\cB_{2,p}^{(1)}$ of $p$ in $\cB_2^{(1)}$, and we can apply \Cref{lemma blow up + covering stack does the job} to obtain a stack $\cB_{2,p}^{(2)}\to\cB_{2,p}^{(1)}$;
    \item The blow-ups constructed \'etale locally glue (from \Cref{lemma blow up last part}) and give a diagram as below
    \[
    \begin{tikzcd}
    \cB^{(1)}_{2,p} \smallsetminus \{p\} \simeq \cB_{2,p}^{(2)}\smallsetminus E \ar[r] \ar[d] & \cB_{2,p}^{(2)} \ar[d] \\
    \cB_2^{(1)}\smallsetminus\{p\} \ar[r] & \cB^{(2)}_2.
    \end{tikzcd}
    \]
\end{enumerate}
The stack $\cB^{(2)}_2$ is a Deligne-Mumford stack, flat over $\spec(R)$, that is isomorphic to $(B_2)|_{\eta}$ over the generic point $\eta$ of $\spec(R)$ and whose fiber over the closed point of $\spec(R)$ consists of a chain of two stacky $\bP^1$'s and an open subset of a curve, meeting in two stacky points with cyclic stabilizers, as in the figure below.

Again, we can consider the proper transform of $S$ in $\cB_2^{(2)}$: this meets the new exceptional divisor in a (schematic) point $p$, so we can apply again points (1)-(3) above.

Doing this operation $r$ times, we obtain an \emph{iterated twisted blow up} $\cB_2^{(r)}$: this is flat over $\spec(R)$, it is isomorphic to ${B_2}_{|\eta}$ over the generic point, and the closed fiber consists of a chain of $r$ stacky $\bP^1$'s with twisted nodes.

We can glue back $\cB_2^{(r)}$ to $B_1$ along $S$: we denote the resulting stack by $\cN^{(r)}$.
\begin{center}
\includegraphics[scale=0.07]{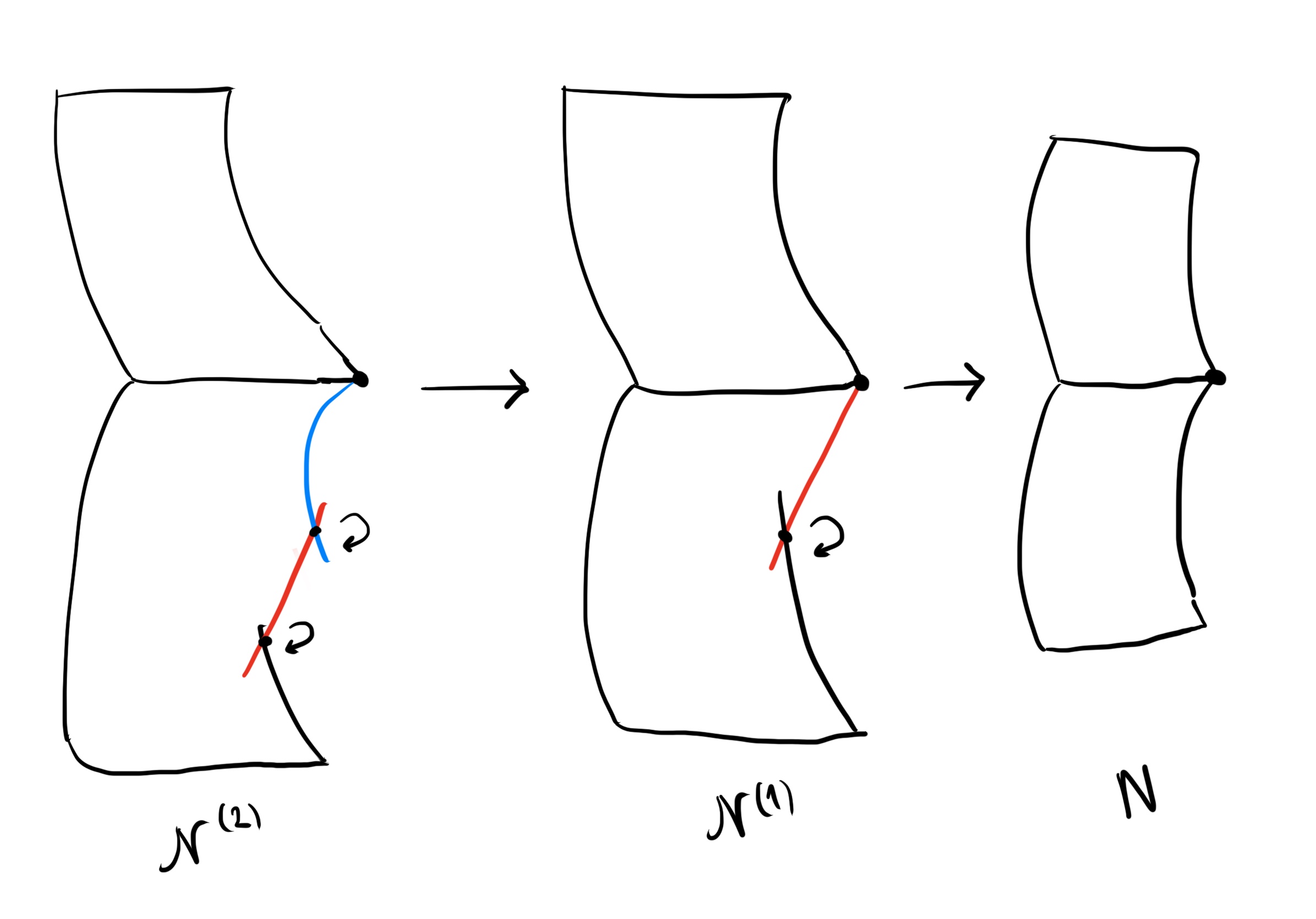} \\ Figure: an iterated twisted blow-up.
\end{center}
The following proposition is a mild generalization of \Cref{proposition val criterion for the nodes}, included to streamline repeated applications of \Cref{proposition val criterion for the nodes}.
\begin{Cor}\label{cor:val criterion for the nodes iterated}
Let $\spec(A)$ be an affine scheme endowed with a $\Gm^n$ action as in \Cref{proposition val criterion for the nodes}. Let $\cN^{(r)}$ be a stack obtained by performing an iterated twisted blow up on $N$, as explained above. 
Assume that there is a diagram as follows:
$$\xymatrix{\cN^{(r)} \smallsetminus \{p\} \ar[r]^-b \ar[d] & [\spec(A)/\Gm^n]\ar[d]\\ N \ar[r]^-{a}  & \spec(A^{\Gm^n}).}$$
By composing with the map $[\spec(A)/\Gm^n]\to \cB\Gm^n$, we have $n$ line bundles on $\cN^{(r)}\smallsetminus \{p\}$, and assume that the line bundles corresponding to the first $k-1$ factors of $\cB\Gm^n$ extend to $\cN^{(r)}$, but the remaining $n-k+1$ do not.

Then we can perform a birational transformation $\rho:\cN^{(r+1)}\to \cN^{(r)}$ which:
\begin{itemize}
    \item is an isomorphism on $\rho^{-1}(\cN^{(r)}\smallsetminus \{p\})$;
    \item $\cN^{(r+1)}\to \spec(R)$ is flat and the closed fiber is a nodal twisted curve made of $r+1$ irreducible components which are stacky $\bP^1$'s, and $r$ twisted nodes;
    \item the map $\rho$ contracts a single (stacky) curve $\cP^1$ whose coarse space is $\bP^1$;
    \item there is a commutative diagram as follows, where $\alpha$ is representable and we denote again by $p$ the intersection of the proper transform of $S$ with the closed fiber:
    $$\xymatrix{\cN^{(r+1)}\smallsetminus \{p\}\ar[r]^-\alpha \ar[d] & [\spec(A)/\Gm^n]\ar[d]\\ N \ar[r]^-a  & \spec(A^{\Gm^n});}$$
    \item the line bundle corresponding to the $k$-th factor of $\cB\Gm^n$ lifts to a line bundle  on $\cN^{(r+1)}$ with degree $\frac{1}{d}$ on the stacky $\cP^1$, and the line bundles corresponding to the first $k-1$ factors of $\cB\Gm^n$ still lift as well.
\end{itemize} 

The construction of $\cN^{(r+1)}$ commutes with \'etale base change.

\end{Cor}
\begin{Oss'}
The stack $\cN^{(r+1)}$ in the statement of \Cref{proposition val criterion for the nodes} is obtained by performing an $(m,d)$-blow up. As in \Cref{proposition val criterion for the nodes}, the integers $m$ and $d$ are determined by the morphism $N\smallsetminus\{p\} \to [\spec(A)/\Gm^n]$.
\end{Oss'}
The proof of \Cref{cor:val criterion for the nodes iterated} essentially relies on the following.

\begin{Lemma}\label{lemma estensere line bundles etale locally suffices}
Assume that $\cX$ is a DM stack which is $S_2$ and of dimension 2, and let $p$ be a closed point and $U$ its complement. Assume that one has a line bundle $L$ on $U$, that extends to a line bundle on an \'etale neighbourhood $j:V\to \cX$ of $p$. Then it extends to a line bundle on $\cX$.
\end{Lemma}
\begin{proof}
Up to working on a chart, we can assume that $\cX = X$ is a scheme. If $i:U\to X$ is the inclusion, then $i_*L$ is $S_2$.
The pull-back of $S_2$ sheaves under \'etale morphisms is still $S_2$, so $j^*i_*L$ is $S_2$. As it agrees with a line bundle away from the preimage of $p$, the sheaf $j^*i_*L$ is a line bundle, so its stalk at the preimage of $p$ is one dimensional. So the stalk of $i_*L$ at $p$ is one dimensional, so $i_*L$ is a line bundle as desired.
\end{proof}
\begin{proof}[Proof of \Cref{cor:val criterion for the nodes iterated}]
First, we perform an additional iteration of the twisted blow-up.
Write $\cB_1^{(r)}$ and $\cB_2^{(r)}$ for the two components of $\cN^{(r)}$. We define $\cB_2^{(r+1)}$ by performing an $(m,d)$-blow up of $\cB_2^{(r)}$ centered at $p$ as explained above (the values of $m$ and $d$ and the branch on which to perform the blow up are determined exactly as in \Cref{proposition val criterion for the nodes}), and then we define $\cN^{(r+1)}$ by gluing together $\cB_1^{(r)}$ and $\cB_2^{(r+1)}$ along $S$.

Let $\varphi_k:\cN^{(r)}\smallsetminus\{p\}\to \cB\Gm^n\to \cB\Gm$, where the last map is the projection on the $k^{\rm th}$ factor, and let $\cN_p^{(r)}$ be an \'etale neighbourhood of $p$, so that we have the restricted morphism $\varphi_{k,p}:\cN_p^{(r)}\smallsetminus\{p\}\to \cB\Gm$. Applying \Cref{proposition val criterion for the nodes} we can lift this map to $\cN_p^{(r+1)}$, the preimage of $\cN_p^{(r)}$ in $\cN^{(r+1)}$, and this extension is unique since the construction in \Cref{proposition val criterion for the nodes} commutes with \'etale base change. We also have the restriction $\varphi_{k,0}:\cN^{(r)}\smallsetminus\{p\}\to \cB\Gm$, hence we obtain a commutative diagram
$$\xymatrix{\cN^{(r)}_p\smallsetminus\{p\}\simeq \cN^{(r+1)}_p\smallsetminus E \ar[r] \ar[d] & \cN_{p}^{(r+1)} \ar[d] \ar[dr]^{\varphi'_{k,p}} & \\ \cN^{(r)}\smallsetminus\{p\} \ar[r]\ar@/_1.5pc/[rr]_{\varphi_{k,0}}  & \cN^{(r+1)} \ar@{..>}[r]^{\varphi_k'}& \cB\Gm.}$$

The dotted arrow exists on the codimension one points of $\cN^{(r+1)}$ since those are the codimension one points of $\cN^{(r)}$ (where $\varphi_{k,0}$ is defined), and the generic point of the exceptional divisor for the twisted blow-up (where $\varphi_{k}'$ comes from the line bundle $\cI$ of \Cref{lemma blow up + covering stack does the job}). Then the dotted arrow exists from \Cref{lemma estensere line bundles etale locally suffices}. The other statements are straightforward to check using \Cref{proposition val criterion for the nodes}.
\end{proof}

\section{Proof of the existence part of the valuative criterion}\label{section proof val criterion}
This section is devoted to the proof of our main theorem; the argument proceeds as follows. 

We first deal with the case where the generic fiber is smooth. For this case, using the properness of Kontsevich spaces,  we find a limit of the map on good moduli spaces.
This is however only an intermediate 
step, as one still needs to lift the map from $M$ to $\mathcal{M}$, up to adding some stacky structure along the nodes of the closed fiber.
We will lift this map
up to adding some stacky structure along the nodes of the closed fiber, in two steps. First, we will lift it along the generic points of the closed fiber using \Cref{teorema fix}. Then we will use \Cref{lemma smooth purity} to extend it along the closed points of the closed fiber whose local ring is normal. 

\medskip

To deal with a nodal generic fiber, we reduce to the smooth case by taking the normalization of the generic fiber, and we use the fact that a nodal curve is a pushout. Therefore, we first extend the map to the normalization of the generic fiber,  using the case of a smooth generic fiber. After that, using \Cref{proposition val criterion for the nodes}, we show that we can extend the pushout data up to performing some $(m,d)$-blowups. Finally we will use some results from \Cref{subsection contracting rational curves} on contracting (stacky) rational curves to obtain point (3) in \Cref{def twisted map}.

\begin{Teo}\label{theorem existence part val criterion}
Let $R$ be a DVR with algebraically closed residue field, let $\eta$ (resp. $x$) be the generic (resp. closed) point of $\spec(R)$. Let $\cM=[X/G]$ be a quotient stack by a linear algebraic group $G$ having a projective good moduli space $M$.  Assume we are given a pointed twisted map $(\pi:\cC_\eta\to \eta, \phi:\cC_\eta\to \cM,\{\sigma_i : \Sigma_{i,\eta} \to \cC_\eta\})$ over $\eta$.

Then, up to replacing $\spec(R)$ with a ramified cover, we can extend the pointed twisted map over $\eta$ to a pointed twisted map \[(\pi:\cC_R\to \spec(R), \phi:\cC_R\to \cM,\{\sigma_i : \Sigma_{i,R} \to \cC_R\})\text{ over }\spec(R).\] Moreover, when the generic fiber has no destabilizing $\bP^1$ and $G=\GL_n$, the length of a chain of destabilizing $\bP^1$s of the coarse moduli space of the closed fiber of $\cC_R\to\spec(R)$ is at most $n$. 
\end{Teo}

The proof proceeds in three stages. We first handle the case where $\mathcal{C}_\eta$ is smooth in \Cref{prop_smooth_case}, then the case where its coarse space is a destabilizing $\mathbb{P}^1$ in \Cref{lemma_val_crit_dest_P1}. The general case then follows by taking the normalization of $\mathcal{C}_\eta$ and reducing to these two cases on each irreducible component.

We begin with the following observation:
\begin{Oss}
    In \Cref{theorem existence part val criterion} we can assume $G=\GL_n$: indeed, by hypothesis we have $G\hookrightarrow \GL_n$, hence $\cM=[X/G]=[Y/\GL_n]$, where $Y$ is the quotient of $X\times\GL_n$ by $G$. We will make this assumption throughout.
\end{Oss}

\begin{Prop}\label{prop_smooth_case}
     In addition to the assumptions of \Cref{theorem existence part val criterion}, assume also that $\cC_\eta$ is smooth and its coarse space, together with the coarse spaces of $\Sigma_{i,\eta}$, induce a pointed Kontsevich-stable map to $M$.

     Then the results of \Cref{theorem existence part val criterion} hold. Moreover, in this case, one can arrange the closed fiber of \[(\pi:\cC_R\to \spec(R), \phi:\cC_R\to \cM,\{\sigma_i : \Sigma_{i,R} \to \cC_R\})\] not to have any destabilizing $\bP^1$s.
\end{Prop}

\begin{proof}
Let $S_\eta $ (resp. $C_\eta$) be the coarse moduli space of $\Sigma_{i,\eta}$ (resp. $\cC_\eta$).
 From properness of Kontsevich stable maps, the map $(C_\eta,S_\eta)\to M$ can be extended to a map \[(C_R, S_R)\to M\text{ over }\spec(R)\] up to replacing $R$ with a ramified cover of it. We then perform a root stack $\cC'_R\to C_R$, ramified over $S_R$, so that its restriction to the generic fiber $(\cC'_R)_{|\eta}$ is isomorphic to $\cC_\eta$.

We first extend the map $(\cC'_R)_{|\eta}\cong \cC_\eta\to \cM$ to the generic points of the closed fiber using \Cref{teorema fix}. If $\xi$ is one such generic point, $\cO_{\cC'_R,\xi}$ is a DVR and the induced map $\spec(\cO_{\cC'_R,\xi})\to\spec(R)$ satisfies the assumptions of \Cref{teorema fix} since the fibers of $\cC'_R\to \spec(R)$ are smooth in a neighbourhood of $\xi$. Moreover, as $\cM=[X/G]$ is a quotient stack, it has a Zariski open cover given by substacks of the form $[\spec(A)/G]$ and hence we have a factorization
\[
\xymatrix{
    & \left[\spec(A\otimes_{A^G}\cO_{\cC'_R,\xi})/G\right]\ar[r] \ar[d]^{\mu'} & \left[\spec(A)/G\right] \ar[r] \ar[d]^{\mu}& \cM \ar[d] \\\spec(\cO_{\cC_\eta})\ar[r] \ar[ru]&
    \spec(\cO_{\cC'_R,\xi}) \ar[r] & \spec(A^G) \ar[r] & M}
\]
where $\mu$ (and therefore also $\mu'$) is universally closed from \cite[Theorem A.8]{AHH}.

Therefore \Cref{teorema fix} applies, and we have a map $\spec(\cO_{\cC'_R,\xi})\to \cM$ extending our map on $(\cC'_R)_\eta$, up to replacing $R$ with a (potentially ramified) cover of it. This allows us to extend the map $\cC'_\eta \to \cM$ to the generic points of the closed fiber of $\cC'_R\to \spec(R)$. We can now spread out this map: up to removing a finite set of points $\{p_1,...,p_n\}$ of the closed fiber,  we get an extension \[\cC_R'\smallsetminus \{p_1,...,p_r\}\to \cM.\] 

Observe that $\cC'_R$ has cyclic quotient singularities, so one can apply \Cref{lemma_generalization_purity_tsm1} to construct an extension $\cC_R\to \cM$. One can check that this procedure gives a pointed twisted map over $\spec(R)$, with coarse moduli space $C_R$.\end{proof}

\begin{Lemma}\label{lemma_val_crit_dest_P1}
    In addition to the assumptions of \Cref{theorem existence part val criterion}, assume also that $\cC_\eta$ has only two markings $\Sigma_1$ and $\Sigma_2$ and is such that its coarse space is a destabilizing $\bP^1$.

     Then the results of \Cref{theorem existence part val criterion} hold. Moreover, in this case, one can arrange the closed fiber of \[(\pi:\cC_R\to \spec(R), \phi:\cC_R\to \cM,\{\sigma_i : \Sigma_{i,R} \to \cC_R\})\] to have coarse moduli space isomorphic to $\bP^1$.
\end{Lemma}
\begin{proof}
The coarse space of $\cC_\eta$ is $\bP^1_\eta$, and the two gerbes give two points $p_1,p_2$ in $\bP^1_\eta$. Those are sections of $\cP^1_\eta\to \eta$, so
we have a map $\eta\to M$, and in turn as $M$ is proper, a map $\spec(R)\to M$. We can extend $(\bP^1_\eta, p_1,p_2)\to\eta$, as $\bP^1_\eta = \operatorname{Proj}(K(R)[X_0,X_1])$, to the trivial family over $\spec(R)$, to get $(\bP^1_R,S_1,S_2)$, where $S_1$ and $S_2$ are obtained by taking the closure of $p_1$ and $p_2$. Observe that, up to acting on $\bP^1_ \eta$ via $\operatorname{Aut}((\bP^1_\eta,p_1,p_2))\simeq \Gm(K(R))$, we can assume that $S_1$ and $S_2$ remain disjoint in the limit.

If $\bmu_{k_i}$ is the stabilizer of $\Sigma_i$, we can take a root stack of $\bP^1_R$
along $S_1$ and $S_2$ with orders $k_1$ and $k_2$ respectively. If we denote this root stack by $\cP^1_R$, then $\cP^1_\eta \cong \cC_\eta$, and we have the following diagram of solid arrows:

$$\xymatrix{\cP^1_\eta \ar[d]\ar[r]\ar@/^1pc/[rr] & \cP^1_R\ar[d]\ar[rd] \ar@{..>}[r]& \cM\ar[d]\\ \eta\ar[r] & \spec(R)\ar[r]& M.}$$
We now proceed as in \Cref{prop_smooth_case}, and we obtain the desired extension.
\end{proof}
We are finally ready to prove \Cref{theorem existence part val criterion}, which is the main theorem of this paper.
\begin{proof}[Proof of \Cref{theorem existence part val criterion}] Let $(\cC_\eta^n,\Sigma_\eta + \Theta_\eta)\to \cM$ be the normalization of $\cC_\eta \to \eta$, where $\Theta_\eta$ is the preimage of the nodal points. Up to replacing $R$ with a cover of it, we can assume that the connected components of $\Theta_\eta$ are gerbes over $\eta$. Therefore from \Cref{prop_smooth_case} and \Cref{lemma_val_crit_dest_P1} we can extend the twisted pointed map $(\cC_\eta^n,\Sigma_\eta + \Theta_\eta)\to \cM$ over $\spec(R)$ to get maps $((\cC')^n_R,\Sigma_R + \Theta_R)\to \cM$. 

Let $\Theta_{i,\eta}$ and $\Theta_{i,R}$ be the connected components of $\Theta_\eta$ and $\Theta_R$, indexed by $i\in I$. Recall that $\cC_\eta$ is obtained from $\cC_\eta^n$ via the gluing data $\Theta_{i,\eta}\simeq \Theta_{\sigma(i),\eta}$, where $\sigma:I\to I$ is an involution with no fixed points. By construction, the gerbes $\Theta_{i,R}$ are obtained as root stacks over their image in the coarse space of $(\cC')^n_R$; from this, it follows easily that the gluing data over the generic point extend to isomorphisms \[\Psi_i\colon \Theta_{i,R}\simeq \Theta_{\sigma(i),R}.\] Therefore, by gluing $(\cC')^n_R$ using $\Psi_i$, we obtain a new family of curves $\cC'_R$. 

The following diagram summarizes the situation. The solid arrows are the maps already constructed; the dotted arrow $\phi$ is the extension we wish to produce. The family $\mathcal{C}'_R$ is obtained by gluing $(\mathcal{C}')^n_R$ along the extended gluing data $\Psi_i$, and the problem reduces to extending $\phi$ at the non-smoothable nodes of $\mathcal{C}'_R$, after performing some twisted blow-ups as in \Cref{cor:val criterion for the nodes iterated}.

\[\xymatrix{\cC_\eta^n \ar[d]\ar[r] & (\cC')^n_R\ar[r] \ar[d] & \cM\ar[d] \\
\cC_\eta \ar[r] \ar[d] \ar[rru] & \cC'_R\ar[r] \ar@{..>}[ru]_\phi \ar[d] & M\\\eta \ar[r] & \spec(R). & }\]


Therefore, let us fix a node $\Theta_{i,R}$ which is not an isolated singularity of $\cC'_R$. Let then $T_{i,R}\to C'_R$ be the map $\Theta_{i,R}\to \cC'_R$ on coarse spaces. As $T_{i,R}\to \spec(R)$ is an isomorphism, we have a section $s:\spec(R)\to T_{i,R}\to M$. As $\cM=[X/\GL_n]$ is a quotient stack, there is a ring $A$ with an action of $\GL_n$ on it such that there is a Zariski open morphism $\spec(A^{\GL_n})\to M$ and an isomorphism $\spec(A^{\GL_n})\times_M\cM\cong [\spec(A)/\GL_n]$. We can choose the map $\spec(A^{\GL_n})\to M$ to be a neighbourhood of $s(x)$, namely a neighbourhood of the image of the closed point via $s$.

In what follows we will adopt the following notation: we denote by $p_i$ the closed point of $\Theta_{i,R}$ and we use the symbol $\cN_i$ to denote a small enough \'{e}tale neighbourhood of $p_i$, which is obtained by gluing two smooth irreducible components $B_1$ and $B_2$ along a closed subscheme $S$, as in \Cref{subsection_iterated_twisted_blowup}.

We will now proceed in three steps. First, we will collect some numerical invariants $d,m_1,...,m_n$ that depend only on the neighbourhood of $s(x)$ we choose and on the map $\cN_i\smallsetminus \{p_i\}\to \cM$; second, we will perform some twisted blow-ups and, using the uniqueness of the extension in \Cref{cor:val criterion for the nodes iterated} once the numerical invariants are chosen. Finally, we will apply \Cref{theorem contracting stacky P1s} to guarantee that the resulting limit satisfies condition (3) in \Cref{def twisted map}.


\textbf{Step 1}  Consider the inclusion of the diagonal torus $(\Gm)^n\to \GL_n$.
We first show that the map \[\cN_i\smallsetminus\{p_i\}\to [\spec(A)/\GL_n]\text{ lifts to }\cN_i\smallsetminus\{p_i\}\to [\spec(A)/(\Gm)^n],\] and we collect some numeric invariants $d,m_1,...,m_n$.

As \[[\spec(A)/(\Gm)^n]\cong [\spec(A)/\GL_n]\times_{\cB\GL_n} (\cB\Gm)^n\] and from the universal property of the fiber product, it suffices to check that the vector bundle induced by \[\cN_i\smallsetminus\{p_i\}\to [\spec(A)/\GL_n]\to \cB\GL_n\] splits as a sum of line bundles. If we denote by $B_1$ and $B_2$ the two branches of $\cN_i$, the data of a map $\cN_i\smallsetminus\{p_i\}\to \cB \GL_n$ is the data of two vector bundles $E_j\to B_j$ together with an isomorphism $\phi:(E_1)_{|S_\eta}\to(E_2)_{|S_\eta} $, where $S$ is the singular locus of $\cN_i$ (and which is isomorphic to $\spec(R)$).
Up to choosing two isomorphisms $\alpha:\cO_{B_1}^{\oplus n}\to E_1$ and $\beta:\cO_{B_1}^{\oplus n}\to E_1$, the map $\phi$ corresponds to a matrix $A$ with entries in the fraction
field of $R$. If $\pi$ is a uniformizer of $R$ , there is an $\ell\ge 0$ such that $\pi^\ell A$ has coefficients in $R$, therefore from Smith's normal form we can find two matrices $D_1'$ and $D_2'$ with entries in $R$ such that the product \[D_1'\pi^\ell AD_2'=\Delta\] is a diagonal matrix.
As $S\to B_j$ is a closed embedding of affine schemes, if $B_j = \spec(\cB_j)$, we can find two matrices $D_1$ and $D_2$ with entries in $\cB_1$ and $\cB_2$ respectively which restrict to $D_1'$ and $D_2'$. Those remain invertible as their determinant is not 0 when restricted to the closed point of $S$ (and also of $\spec(\cB_j)$. Therefore if we change basis on $E_j$ according to $D_j$, we can assume that the morphism $\phi$ is a diagonal matrix.

Moreover, the diagonal entries are of the form \[(u_1 \pi^{m_1},...,u_n\pi^{m_n}),\text{ where }u_j\text{ are units, and }m_j\le m_{j+1}\text{ for }j<n.\] 
From the uniqueness properties of Smith's normal form, the integers $m_j$ are uniquely determined.

To determine the integers $d_1,\ldots,d_n$, we look at the action of $\Gm^n$ on $\spec(A)$: the $k^{\rm th}$ factor of $\Gm^n$ defines a grading on $A$, and we set $d_k$ to be an integer such that the graded ring $A$ is generated in degree $[-d_k,d_k]$.

\textbf{Step 2}: We extend the map $\cN_i\smallsetminus\{p_i\}\to [\spec(A)/(\Gm)^n]$. We start by extending the induced morphism $\cN_i\smallsetminus\{p_i\}\to \cB\Gm^n$. 

For this, we first apply \Cref{proposition val criterion for the nodes}, thus obtaining a stack $\cN_i^{(1)}$ where we can extend the composition $\cN_i\smallsetminus\{p_i\}\to \cB\Gm^n\to \cB\Gm$, the last map being the projection on the first factor. The stack $\cN_i^{(1)}$ is obtained by performing an $(m_1,d_1)$-blow up on one branch, and then glue the two branches back together.

We can then apply \Cref{cor:val criterion for the nodes iterated} to produce a stack $\cN_i^{(2)}$ together with an extension of the composition $\cN_i\smallsetminus\{p_i\}\to \cB\Gm^n\to \cB\Gm^2$, the last map being the projection on the first two factors. Again, the stack $\cN_i^{(2)}$ is obtained by performing an $(m_2,d_2)$-blow up on the second branch and then gluing back together the two branches.

After applying \Cref{cor:val criterion for the nodes iterated} enough times, we obtain a stack $\cN_i^{(n)}$ which is birational to $\cN_i$, its exceptional divisor consists of a chain of stacky $\cP^1$'s, with twists in the nodes, and more importantly it has a morphism to $\cB\Gm^n$ that extends the one defined on $\cN_i\smallsetminus\{p_i\}$.

We can now apply \Cref{rmk:extension} to obtain a map $\cN_i^{(n)}\to[\spec(A)/\Gm^n]$. Composing this map with the morphism $[\spec(A)/\Gm^n]\to [\spec(A)/\GL_n]$, we finally obtain $\cN_i^{(n)} \to [\spec(A)/\GL_n] \to \cM$.

We can repeat the process above for every closed non smoothable node $p_i$ on the central fiber of $\cC_R'\to\spec(R)$, thus producing several $\cN_i^{(n)}$ for every $i$, and a morphism
 $\sqcup_{i} \cN_{i}^{(n)} \longrightarrow \cM.$

Observe that the complement in $\cN_i^{(n)}$ of the exceptional divisor is by construction isomorphic to $\cN_i\smallsetminus\{p_i\}$, and $\cN_i$ is an \'etale neighbourhood of $p_i$. The construction of $\cN_i^{(n)}$ are uniquely determined by the integers $d, m_1,...,m_n$, so they commute with \'etale base change. Then we can glue them to get a twisted curve $\cC''_{i,R}$,i.e. we get the following diagram, where the vertical maps are \'etale neighbourhoods
\[
\begin{tikzcd}
\sqcup_{i} \cN_i\smallsetminus\{p_i\} \ar[r] \ar[d] & \sqcup_{i} \cN_i^{(n)} \ar[d] \ar[dr] & \\
\cC'_R\smallsetminus (\sqcup_i p_i) \ar[r] & \cC''_{i,R}\ar[r] & \cM.
\end{tikzcd}
\]
The morphism $\sqcup_i \cN_i^{(n)}\to \cM$, which extend $\sqcup_i \cN_i\smallsetminus\{p_i\}\to\cM$, are uniquely determined by the integers $d, m_1,...,m_n$, the map $\cC'_R\smallsetminus (\sqcup_i p_i)\to \cM$, and a choice of a maximal torus in $\GL_n$ (the diagonal one) so they descend to $\cC''_{i,R}\to \cM$ as in the diagram above. In other terms, after performing twisted blow ups at the non-smoothable nodes in the central fiber of $\cC_R'\to\spec(R)$, we are able to extend the map.


\textbf{Step 3}: the goal of this last step is contracting stacky $\cP^1$s that factor via $\cB\bmu_k\to \cM$ for some integer $k$.

In fact, at the end of step (2), we have proved that we can extend the morphism $C_\eta \to \cM$ the map whose domain is a new family of curves $\cC_R''$. We are then left with showing that if the generic fiber is a twisted map, then we can assume that also the special one is a twisted map. More precisely, we claim that we can perform a further birational modification $\cC_R''\to \cC_R$ over $\spec(R)$ so that the morphism $\cC''_R \to \cM$ descends to a morphism $\cC_R\to \cM$ and the new family of twisted curves $\cC_R\to \spec(R) $ satisfies also condition (3) in Definition \ref{def twisted map}. This claim follows from \Cref{theorem contracting stacky P1s}.
\end{proof}

\section{Contracting stacky projective lines on a stacky surface}\label{subsection contracting rational curves}
The goal of this section is to collect a few results, analogous to those of Artin \cite{Artin_contracting_curves}, on contracting rational (stacky) curves on 2-dimensional Deligne-Mumford stacks. We focus on the case where the coarse moduli spaces of our surface singularities are $A_n$ singularities, as those are the singularities which are of interest for one parameter families of curves.  Therefore we begin with the following. 
\begin{Lemma}\label{Lemma extend the action around a node to a resolution}
Consider the singularity $X=\spec(k[x,y,z]/(xy-z^a))$, with $a>1$,  and assume that $\bmu_k$ acts on $k[x,z,y]/xy-z^a$ by sending $\xi*x= \xi x$, $\xi*y= \xi^{-1} y$ and $\xi*z=z$. Consider the blow-up of $X$ at the origin. Then the action extends to the two exceptional divisors, and the action on the two components of the exceptional divisor is balanced. Such an extension is unique.
\end{Lemma}
\begin{proof}First observe that we can extend the action of interest to an action on $\bA^3$. We can extend it to the blow-up of $\bA^3$ at the origin: this is $\operatorname{Proj}(k[x,y,z,u,v,w]/(xu-yv, xw-zv, yw-zu))$ where $x$, $y$ and $z$ have degree zero and $u$, $v$ and $w$ have degree one, and we extend our action by defining $\xi*u=\xi^{-1}u$, $\xi*v=\xi v$ and $\xi*w = w$.  We can take the proper transform of the hypersurface $xy-z^a=0$, which will be the desired blow-up with the desired action extended. The uniqueness follows as both actions agree on a dense open subset, and our blow-up is separated.
\end{proof}
\begin{Oss}\label{Oss extend the action around a node to a resolution}It is easy to check how the action is defined on the blow-up of \Cref{Lemma extend the action around a node to a resolution}. Indeed, the chart where $w=1$ is isomorphic to the spectrum of $k[\frac{u}{w}, \frac{v}{w}, z]/(\frac{v}{w}\frac{u}{w}-z^{a-2}),$ with the action defined as $\xi * \frac{v}{w} = \xi\frac{v}{w}$, $\xi*\frac{u}{w} = \xi^{-1}\frac{u}{w}$ and $\xi*z = z$. The charts where $u=1$ or the one where $v=1$ can be described analogously. 
\end{Oss}

\begin{Prop}\label{prop Artin contractions}
Assume that $\cX$ is a 2-dimensional Deligne-Mumford stack, proper over the spectrum of a DVR with coarse space $X$. Let $\cP^1\subseteq \cX$ be a proper substack with coarse moduli space isomorphic to $\bP^1$. Assume that:
\begin{enumerate}
\item there are two distinct points $p,q\in \cP^1$  with stabilizer $\bmu_k$, such that $\cP^1\smallsetminus \{p,q\}$ is a scheme,
\item the singularities of $\cX$ (resp. $X$) at $p$ and $q$ are of type $A_{m-1}$ and $A_{n-1}$(resp. $A_{km-1}$ and $A_{kn-1}$), 
\item $\cX$ is smooth at every point of $\cP^1\smallsetminus \{p,q\}$ and there is $U\subseteq \cX$ an open neighbourhood of $\cP^1\subseteq U$ and such that $U$ is a scheme in codimension one, and 
\item there is a contraction $f:X\to Y$ whose exceptional divisor is just the coarse moduli space of $\cP^1$, and $Y$ has $A_{n'}$-singularities for an integer $n'$.
\end{enumerate}
Then there is a Deligne-Mumford stack $\cY$ with coarse moduli space $Y$ and with a representable map $\psi:\cX\to \cY$ lifting $f$, which induces isomorphisms $\operatorname{Aut}_\cX(p)\to \operatorname{Aut}_\cY(\psi(p))$ and $\operatorname{Aut}_\cX(q)\to \operatorname{Aut}_\cY(\psi(q))$ .
\end{Prop}
\begin{proof}
If $k=1$ the stack $\cX$ is isomorphic to its coarse space $X$, and we have nothing to prove. We can assume $k\ge 2$. 

Let $C\cong \bP^1$ be the coarse space of $\cP^1$ on $X$ and $C_{Z}$ its proper transform on a minimal resolution $\xi:Z\to X$ of $X$ at $p$ and $q$. 
 We first check that $C_Z$ is not a $(-1)$-curve. The scheme $X$ has Gorenstein singularities, so $K_X.C$ is an integer. Moreover, $C^2<0$ as it is contracted by $X\to Y$. By inversion of adjunction, $(K_X+C)|_C = \omega_C(\operatorname{Diff}_C)$ where $\operatorname{Diff}_C$ is an effective $\bQ$-divisor (see \cite{Kollarsingmmp}*{Definition 2.34, Proposition 2.35 (3) and Definition 4.2},  or \cite{BI}*{Section 2.2} for the definition and properties of $\operatorname{Diff}_C$). As by assumption $X$ is singular at $p,q\in C$, from \cite{Kollarsingmmp}*{Proposition 4.5 (1)} the $\bQ$-divisor $\operatorname{Diff}_C$ is supported at two points. Now, if $(X,C)$ is not an lc pair, from \cite{Kollarsingmmp}*{Proposition 4.5 (2)} the divisor $\operatorname{Diff}_C$ has a point with coefficient more than 1. If instead $(X,C)$ is lc, from \cite{Kollarsingmmp}*{4.4} at each singular point of $X$ on $C$ the coefficient of the different is at least $\frac{1}{2}$. Since $p$ and $q$ are quotients by $\bmu_{km}$ and $\bmu_{nk}$, $$(K_X+C).C = \deg(\omega_C(\operatorname{Diff})) \ge -2 + \frac{1}{2} +\frac{1}{2} = -1 $$
with equality only if $(X,C)$ is log-canonical.
As $C^2<0$, we have $K_X.C > -1$. But the singularities of $X$ are Du Val, so $\xi^*K_X = K_Z$, so $K_Z.C_Z>-1$: by adjunction, we have $(K_Z+C_Z).C_Z=-2$, hence the curve $C_Z$ cannot be a $(-1)$-curve. 

Then from the minimality of the resolution $\xi$, there are no $(-1)$-curves contracted by $Z\to Y$, so $Z\to Y$ is a minimal resolution which contracts $km +kn - 1$ rational curves. As $Y$ has $A_{n'}$-singularities, $Y$ has an $A_{k(m+n)-1}$-singularity at the image of $p$ and $q$; we denote it by $z$. 
 
 This singularity is the $\bmu_k$-quotient of an $A_{m+n-1}$-singularity, so let $\cY$ be the stacky surface having this singularity,  with a $\bmu_k$ stabilizer at the singular point, and with a map $\cY\to Y$.  Consider $\cY'\to \cY$ a minimal resolution of singularities around $z$. This is obtained by blowing up the closed singular points $\lfloor \frac{m+n}{2} \rfloor$ times (see Observation \ref{Oss extend the action around a node to a resolution}). Then there is a map $\cY'\to \cX$. Indeed, we can similarly resolve the singularities on $\cX$ by blowing up the singular points, and it is easy to see (for example using \cite{GS17}) that such a resolution is isomorphic to $\cY'$. Therefore we have two birational, proper and representable maps $a:\cY'\to \cY$ and $b:\cY'\to \cX$.
Now, consider $\cX'\to \cX$ the covering stack of $\cX$ around $\cP^1$.   As $\cX'\to \cX$ is an isomorphism away from $\cP^1$,  and $\cX\smallsetminus \cP^1 \cong \cY\smallsetminus \{z\}$ by construction, there is a dense open subset $\cU$ of $\cX'$ that admits a morphism $\cU\to \cY$.  We wish to extend this morphism to $\cX'$ using \cite[Lemma 7.4]{DH18}. So let $\spec(R)\to \cX'$ be a morphism from a DVR, that sends the generic point on $\cU$ (i.e. away from the preimage of $\cP^1$ in $\cX'$). As $\cY'\to \cX$ is proper and representable, and an isomorphism over $\cU$, we can lift \emph{uniquely} the map $\spec(R)\to \cX'$ to a map $\spec(R)\to \cY'$. In turn, by composing it with $\cY'\to \cY$, we have a (unique) map $\spec(R)\to \cY$. Therefore from \cite[Lemma 7.4]{DH18} we have a morphism $\cX'\to \cY$. Consider now $\cX^r\to \cY$ its relative coarse moduli space. The stabilizer points $\operatorname{Aut}_{\cX'}(p)$ and $\operatorname{Aut}_{\cX'}(q)$ are isomorphic to $\bmu_{km}$ and $\bmu_{kn}$, so $\operatorname{Aut}_{\cX^r}(p)$ and $\operatorname{Aut}_{\cX^r}(q)$ are quotients of $\bmu_k$: the composition $\cX'\to \cX^r$ factors via $\cX'\to \cX\to \cX^r$, so there is a map $\cX\to \cX^r$ and in turn a map $\cX\to \cY$.  It is easy to check (for example, using \cite[Lemma 7.2]{DH18}) that the compositions $\cY'\to \cY$ and $\cY'\to \cX\to \cY$ are isomorphic.  As $\cY'\to \cY$ is surjective on automorphism groups, also $\cX\to \cY$ is surjective on automorphism groups. Then it is also injective in automorphism groups as both have automorphisms groups isomorphic to $\bmu_k$ along the points $p,q$ and $z$, and elsewhere they are isomorphic. So $\cX\to \cY$ is representable, as desired. The statement about automorphisms groups now follows since an injective map $\bmu_k\to \bmu_k$ is an isomorphism.
\end{proof}
\begin{Lemma}\label{lemma_pushout}
Let $f\colon \cX\to \cY $ be a morphism of 2-dimensional separated Deligne-Mumford stacks such that:
\begin{enumerate}
    \item there is a closed substack $\cP^1\subseteq \cX$ which is the root stack of $\bP^1$ at two points $p,q$, and a commutative diagram as below with horizontal maps which are closed embeddings 
    \[
    \xymatrix{\cP^1\ar[r] \ar[d] & \cX\ar[d] \\ \cB \bmu_k\ar[r] &\cY}
    \]
   \item in a neighbourhood of $\cP^1$ we have that $\cX$ has finite quotient singularities, $\cX$ and $\cY$ are normal, and $\cX$ is smooth at every point on $\cP^1\smallsetminus\{p,q\}$, and
   \item the map $\cX\to \cY$ is an isomorphism away from $\cP^1$.
\end{enumerate}
Then the square above is a pushout in algebraic stacks.
\end{Lemma}
In particular, the $\cY$ constructed in \Cref{prop Artin contractions} is a pushout.
Before giving the proof, let us state and prove a technical lemma on cohomology of line bundles on root stacks of $\bP^1$.
\begin{Lemma}\label{lemma:cohomology root}
    Let $\pi\colon\sP^1\to \bP^1$ be a root stack of order $a$ (respectively $b$) at two points $p$ and $q$. Then for any line bundle $L$ on $\sP^1$ of positive degree we have $\oH^1(\sP^1,L)=0$.
\end{Lemma}
\begin{proof}
    The root stack of order $a$ at $p$ has Picard group generated by $\cO(\frac{1}{a}p)$, with $\cO(\frac{1}{a}p)^{\otimes a} =\cO(p)$. By taking a root of order $b$ at $q$, the Picard group gains another generator $\cO(\frac{1}{b}q)$, whose $b$-tensor power is equal to $\cO(q)$.

    Therefore, there exist integers $0\leq u<a$, $0\leq v< b$ and $c$ such that $L=\cO(\frac{u}{a}p + \frac{v}{b}q + c )$. As $\sP^1$ is a Deligne-Mumford stack whose coarse space is $\bP^1$, we have a well-defined notion of (rational) degree, and in particular $\deg(L)=\frac{u}{a} + \frac{v}{b} + c $. This being positive, we deduce that $c>-2$. 

    From the Leray spectral sequence for morphisms of algebraic stacks \cite{stacks-project}*{Lemma 0782}, we deduce that if $\oH^0(\bP^1, R^1\pi_*L)=\oH^1(\bP^1,\pi_*L)=0$, then $\oH^1(\sP^1,L)=0$. Vanishing of $R^1\pi_*L$ follows from $\pi_*$ being an exact functor: indeed, $\pi\colon \sP^1\to\bP^1$ is also a good moduli space. Observe that if $f\colon X \to \bP^1$ is a root stack on $\bP^1$ along a divisor $D$ of order $d$, $M$ is a line bundle on $\bP^1$ and and $N=\cO_X(\frac{u}{d}D)$ with $0\le u <d$, then $f_*(f^*M \otimes N)=M$. From this it follows that $\pi_*L=\cO(c)$, and its first cohomology group vanishes because $c>-2$.
\end{proof}
\begin{proof}[Proof of \Cref{lemma_pushout}]
The desired statement follows from \cite{arena2023criterion}*{Proposition 4.5}. Indeed, let us denote by $\sX$ the canonical covering stack of $\cX$ around $\cP^1$, which exists as $\cX$ has finite quotient singularities in a neighbourhood of $\cP^1$, and let $\sP^1$ be the stack $(\cP^1\times_\cX\sX)^{red}$, where the superscript red stands for reduced structure. Let finally $\sI$ be the ideal sheaf of $\sP^1$ in $\sX$. Let us denote $\rho\colon \sX \to \cY$ the composition of the contraction $\sX\to\cX$ given by \Cref{prop Artin contractions} with $\pi$.

First observe that $\sI$ is a line bundle as $\sX$ is smooth, and $\oH^1(\sP^1,\sI^n/\sI^{n+1})=\oH^1(\sP^1,(\sI/\sI^2)^{\otimes n})=0$ for every $n\ge 1$, thanks to \Cref{lemma:cohomology root} and the fact that $\sX\to \cX$ is an isomorphism on $\cP^1\smallsetminus\{p,q\}$. 
{We check that $\lim_n \oH^0(\sX,\cO_{\sX}/\sI^n) = \lim_m \oH^0(\cY,\cO_{\cY}/\mathcal{J}^m)$ where $\mathcal{J}$ is the ideal sheaf of $\cB\bmu_k$ in $\cY$. Consider}

\[
0\to \sI^n\to \cO_{\sX}\to \cO_{\sX}/\sI^n\to 0 \quad \overset{\text{via }\rho_*}{\Longrightarrow} \quad 0\to \rho_*\sI^n\to \cO_{\sY}\to \rho_*(\cO_{\sX}/\sI^n)\to R^1\rho_*\sI^n,
\]
{where $\rho_*\cO_{\sX}=\cO_\cY$ as both $\cX$ and $\cY$ are normal and $\rho$ is an isomorphism on a dense open substack.}

{We claim that $R^1\rho_*\sI^n=0$. As $\rho$ is an isomorphism on the preimage of the complement of $\cB\bmu_k$, it is enough to prove that the formal completion of $R^1\rho_*\sI^n$ at $\cB\bmu_k$ is zero: in fact, this would imply that $R^1\rho_*\sI^n$ and the zero constant sheaf have the same completion at $\cB\bmu_k$, hence by Artin's approximation theorem they are isomorphic \'{e}tale-locally.} 

{By the theorem on formal functions for Deligne-Mumford stacks \cite{AV_compactifying}*{Theorem A.3.1}, the completion of $R^1\rho_*\sI^n$ is isomorphic to the limit of $\oH^1(\sP^1_j,\sI^n|_{\sP^1_j})$ over $j$, where $\sP^1_j$ is the closed substack with ideal sheaf $\sI^j$. We prove that these groups vanish for every $n\geq 1$ by induction on $j$, with base case $j=1$. The short exact sequence}
\[ 0\to \sI^{j}/\sI^{j+1} \otimes \sI^n\to \cO_{\sP^1_{j+1}}\otimes \sI^n \to \cO_{\sP^1_j}\otimes \sI^n\to 0\]
{induces the exact sequence $\oH^1(\sP^1_j,\sI^j/\sI^{j+1}\otimes \sI^n|_{\sP^1_{j}}) \to \oH^1(\sP^1_{j+1}, \sI^n|_{\sP^1_{j+1}}) \to \oH^1(\sP^1_j,\sI^n|_{\sP^1_{j}})$. The last term vanishes by induction hypothesis, and the first term vanishes from \Cref{lemma:cohomology root} as 
$\sI^j/\sI^{j+1}\otimes \sI^n|_{\sP^1_{j}}=(\sI/\sI^2)^{\otimes (n+j)}$.}

{Then $\{\rho_*\sI^n\}_n$ is a sequence of ideal sheaves of $\cO_\cY$ supported at $\cB\bmu_k$. It eventually contains $\mathcal{J}$ so the two sequences $\{\rho_*\sI^n\}_n$ and $\{\mathcal{J}^n\}_n$ are both cofinal. In particular, $\lim_n \oH^0(\cO_{\sX}/\sI^n) = \lim_n\oH^0(\cO_{\sY}/\rho_*\sI^n) =  \lim_m \oH^0(\cO_{\cY}/\mathcal{J}^m)$.
Therefore, the stack $\cY$ is the push-out of $\sP^1\to \sX$ and $\sP^1\to \cB \bmu_k$.

{So, assume we are given two morphisms $g\colon\cX\to \cM$ and $\cB \bmu_k\to \cM$ with an isomorphism when restricted to $\cP^1$. Then we can pull them back to $\sX$ and $\sP^1$, and by assumption there is a morphism $\cY\to \cM$, as the latter is a pushout of $\sP^1\to \sX$ and $\sP^1\to \cB\bmu_k$. We just need to check that the two morphisms $f\colon \cX\to \cY\to \cM$ and $g\colon \cX\to \cM$ are isomorphic. }
As $\cX$ and $\sX$ are isomorphic in codimension one, the two morphisms $f,g$ are isomorphic in codimension one, so they are isomorphic {from \Cref{lemma_two_maps_that_agree_in_codim_2_are_the_same}} as $\cX$ is $S_2$ {and $\cM$ has affine diagonal}.}
\end{proof}

\begin{Prop}\label{prop artin contraction theorems for stacks}
With the assumptions and notations of \Cref{prop Artin contractions}, assume also that there is an algebraic stack $\cM$ and a representable map $\phi: \cX\to \cM$. Assume that there is a factorization $\phi|_{\cP^1}:\cP^1 \to \cM$ as $\cP^1 \to \cB\bmu_d \to \cM$ for a certain $d$. Then there is a representable map $\cY\to \cM$ which factors $\cX\to \cM$, and this map is unique.
\end{Prop}

\begin{proof}
    By \Cref{lemma_pushout} we know that $\cY$ is a push-out of $\cP^1\to \cX$ and $\cX\to\cB\bmu_k$ in the category of algebraic stacks. If we prove that the factorization $\cP^1\to \cB\bmu_d\to\cM$ factors through $\cB\bmu_k$, we are done. 
    
    As in the proof of \Cref{lemma:cohomology root}, the torsion part of the Picard group of $\cP^1$ is generated by $\cO(\frac{1}{k}(p-q))$, and the map $\cP^1\to \cB\bmu_k$ is induced precisely by this line bundle.
    
    Let $L$ be the torsion line bundle associated to the representable morphism $\cP^1\to\cB\bmu_d$. Representability implies that the compositions $\cB\Aut(p)\to\cP^1\to\cB\bmu_d$ and $\cB\Aut(q)\to\cP^1\to\cB\bmu_d$ are both representable. As the Picard group of $\cB\Aut(p)=\cB\bmu_k$ is generated by the restriction of $\cO(\frac{1}{k}p)$, the restriction of $L$ to $\cB\Aut(p)$ is equal to the restriction of $\cO(\frac{1}{k}p)$. The same argument shows that the restriciton of $L$ to $\cB\Aut(q)$ is equal to $\cO(\frac{1}{k}(p-q))$. So $L$ must be equal to $\cO(\frac{1}{k}(p-q))$, hence we have the claimed factorization $\cP^1\to\cB\bmu_k\to\cB\bmu_d$. 
\end{proof}

We are ready to prove the following.
\begin{Teo}\label{theorem contracting stacky P1s}
Let $\cC\to \spec(R)$ be a family of twisted curves over the spectrum of a DVR with closed point $\xi$. Assume that $\cM$ is an algebraic stack with a projective good moduli space and with a representable map $\phi:\cC\to \cM$ which is, over the generic point of $\spec(R)$, a twisted map. Let $\cP^1\subseteq \cC_\xi$ be an irreducible component of $\cC_\xi$, with coarse space isomorphic to $\bP^1$. Assume that there are two (possibly stacky) nodes of $\cC_0$ on $\cP^1$, and assume that the map $\phi|_{\cP^1}:\cP^1\to \cM$ factors via a map $\cP^1\to \cB\bmu_d\to\cM$ for a certain $d$.

Then one can contract $\cP^1$. In other terms, {up to replacing $\spec(R)$ with a ramified cover,} there is a family of twisted curves  $\cC'\to \spec(R)$, with two morphisms $\cC\to \cC'$ and $\cC'\to \cM$ such that the composition $\cC\to \cC'\to \cM$ is isomorphic to $\phi$ and such that the morphism $\cC\to \cC'$ on coarse spaces is an isomorphism away from $\cP^1$ and it contracts $\cP^1$.
\end{Teo}
\begin{proof}
{Consider the map $\cP^1\to \cB\bmu_d$. Let $m$ be the minimum value such that there is a factorization $\cP^1\to \cB\bmu_m \to \cB\bmu_d$, where the second map comes from the inclusion $\bmu_m\hookrightarrow\bmu_d$. As all maps are representable, there is a $\bmu_m$-cover $C\to \cP^1$ which is a connected scheme. The composition $C\to \cP^1\to \bP^1$ is ramified at at most two points (the nodes on $\cP^1$) and from \Cref{lemma C to p1 ramified at 2 points implies C=p1} we have $C=\bP^1$ and the map is $[a,b]\mapsto [a^m,b^m]$. Then $\cP^1$ has two points with $\bmu_m$-stabilizers.}

{Let $\cN$ be the codimension one irreducible components of the nodal locus of $\cC$. We now assume that either $\cC$ is normal in a neighbourhood of $\cP^1$, or that $\cP^1$ intersects $\cN$ in one point. We first prove the theorem with this additional assumption, then we show that these two cases are the only possible ones.}

{We first produce a contraction $C\to C'$ at the level of coarse moduli spaces, by contracting the coarse moduli space of $\cP^1$ to a point $x\in C'$.}
{One way to do this is as follows: up to taking a ramified cover of $\spec(R)$, we can assume that every irreducible component of the fiber of $C$  over the closed point $\xi = \spec(k_R)$ has a smooth $k_R$-point that deforms to a smooth $K$-point of the generic curve $C_K$. Let us call $D$ the divisor given by the sum of these sections. Up to taking the normalization of $C$ and working on the connected component containing the coarse space of $\cP^1$, we have that either $D$ is relatively ample on the generic fiber $C_K$, or that $D$ is the empty set. In the first case, the linear system $|mD|$ produces the claimed contraction for $m\gg 0$. The second case can only happen if $\cP^1$ intersects $\cN$ two (possibly stacky) points, which we are assuming to not be the case.}

If $\cC$ is normal on a neighbourhood of $\cP^1$, the results of Propositions \ref{prop Artin contractions} and \ref{prop artin contraction theorems for stacks} apply, which give a surface $\cC' $ with a map $\cC'\to \cM$ as desired. One can check from the construction in \Cref{prop Artin contractions} that $\cC'\to \spec(R)$ is a family of twisted curves.

If $\cC$ is not normal in a neighbourhood of $\cP^1$ then $\cP^1$ intersects $\cN$, and as mentioned above we first assume that $\cP^1\cap \cN$ is a single (possibly stacky) point. \Cref{prop Artin contractions} may not apply: for example, $\cC$ may not be a scheme in codimension one in a neighbourhood of $\cP^1$. We begin by introducing some notation.

Let $\cC^\nu\to \cC$ (resp. $(C')^\nu\to C'$) be the normalization of $\cC$ (resp. $C'$). Let $\Delta\subseteq \cC^n$ be the preimage of $\cN\subseteq \cC$, and let $p$ be the image of the composition $\cP^1\to \cC^\nu\to (C')^\nu$. In other terms, as the closed embedding $\cP^1\to \cC$ lifts to $\cP^1\to \cC^\nu$, the point $p\in (C')^\nu$ is the point to which $\cP^1$ is contracted.
Up to replacing $\spec(R)$ with a ramified cover, we can assume that each irreducible component of $\Delta$ and $\cN$ is a gerbe over $\spec(R)$.
 In Step 1 we construct a stacky curve $\cC'$ having $C'$ as coarse moduli space. In Step 2 we construct a morphism $\cC\to \cC'$ which contracts the desired $\cP^1$, and in Step 3 we argue why the morphism $\cC\to \cM$ factors via $\cC\to \cC'\to \cM$. After these three steps, the case in which $\cP^1$ intersects $\cN$ is a single point is finished. We finish the proof by showing that $\cP^1$ cannot intersect $\cN$ in two points.

{\textbf{Step 1.} We construct $\cC'$, via a pushout diagram.}

Let $\cU$ be an open neighbourhood of $p$ in the normalization $(C')^\nu$ of $C'$, where we further added the appropriate stabilizers at the preimage of the nodal locus intersecting $p$ by taking a root at such locus. We glue $\cU$ and $\cC^\nu\smallsetminus\cP^1$ together, via the identification of $\cU\smallsetminus\{p\}$ and $\cC^\nu\smallsetminus\cP^1$; we denote $(\cC')^\nu$ the resulting stack. Observe that, as we added the appropriate stabilizers at the preimage of the nodal locus intersecting $p$, there is a closed embedding $\Delta\to (\cC)^n$.
Finally, the stack $\cC'$ is constructed as the pushout
of $\Delta \to (\cC')^\nu$ and $\Delta \to \cN$, as in the diagram below
\[
\xymatrix{\Delta\ar[r] \ar[d] & (\cC')^\nu\ar[d]\\\cN\ar[r] & \cC'}
\]

Observe that the preimages of $p$ in $(\cC')^\nu$ are smooth, because by construction the twisted curve $\cC$ is, locally at $\cP^1\cap\Delta$, obtained by gluing two smooth stacky surfaces along a smooth divisor. 

\textbf{Step 2.} We show that there is a map $\cC\to \cC'$ which contracts $\cP^1$.

We first argue that there is a morphism $\cC^\nu\to (\cC')^\nu$ among their normalizations. This follows from an explicit computation as in \Cref{lemma blow up + covering stack does the job}: one can check (for example using \cite{GS17} locally at $p$ after taking a canonical covering stack) that the stack $\cC^\nu$ in a neighbourhood of $\cP^1$ is an $(m,d)$-blow up of $(\cC')^\nu$ centered at $p\in(\cC')^\nu$. In particular, there is a contraction $\cC^\nu\to (\cC')^\nu$, so we can again use that $\cC$ is in a push-out diagram with its normalization $\cC^\nu\to \cC$ and the codimension one nodal locus of $\cC$ to construct the desired morphism $\cC\to \cC'$.

\textbf{Step 3.} We show that there is a map $\cC'\to \cM$ such that $\cC\to \cC'\to \cM$ and $\cC\to \cM$ agree.

From \Cref{lemma_pushout} applied with $\cX=\cC^\nu$ and $\cY=(\cC')^\nu$ there is a map $(\cC')^\nu\to \cM$. As $\cC'$ fits in a pushout diagram as above with $(\cC')^\nu$ and the dimension one nodal locus of $\cC$, we produce the desired map $\cC'\to \cM$. By construction, again using that $\cC$ fits in the usual pushout diagram, one can check that the two maps $\cC\to \cM$ and $\cC\to \cC'\to \cM$ agree.

To conclude the proof, we now argue that $\cP^1$ cannot intersect $\cN$ in two points. In other terms, if $\psi:(\cP^1\times\spec(R),\Delta|_{\cP^1\times\spec(R)})\to \cM$ is a representable morphism which generically is a 2-pointed twisted map, then also over the closed fiber is a 2-pointed twisted map. For doing so, we can replace $\cM$ with an open neighbourhood of it containing $\psi(\cP^1\times\spec(R))$, at the cost of dropping the assumption that $\cM$ has a compact good moduli space.
By assumption, $\cP^1$ is the quotient of $\bP^1$ via the $\bmu_m$-action that sends $[a,b]\mapsto [a^m,b^m]$; we denote by $\sigma\colon \bP^1\to \cP^1$ be this quotient by $\bmu_m$.

\textbf{Case $\cM=\cB\GL_n$.}
Let $\spec(\overline{k}_R)\to \spec(R)$ be the closed geometric point. From
\cite[Theorem 2.4]{martens2012variations} the vector bundle $\cV$ coming from $\cP^1_{\overline{k}_R}\to \spec(R)\times \cP^1 \xrightarrow{\psi}\cB\GL_n$ splits as a direct sum of line bundles $\cV = \cL_1\oplus \ldots \oplus \cL_n$. As $\bP^1_{\overline{k}_R}\xrightarrow{\sigma} \cP^1_{\overline{k}_R}\to \cB \bmu_m$ factors as $\bP^1_{\overline{k}_R}\to \spec(k)\to \cB\bmu_m$, the line bundles $\cL_i$ are all torsion line bundles, as they become trivial once pulled back via $\sigma$. Since
$\Pic(\cP^1_{\overline{k}_R})=\Pic(\cP^1)$, namely, the Picard group of the root stack of $\bP^1$ at two points does not depend on the field of definition of $\bP^1$, there is a map $\alpha\colon\cP^1\to \cB\GL_n$ such that the
vector bundle associated to $\psi$, which we will denote by $\cE_1$, is isomorphic, when restricted to $\cP^1_{\overline{k}_R}$, to a vector bundle $\cE_2$ coming from the composition $\spec(R)\times\cP^1\to \cP^1\xrightarrow{\alpha} \cB \bmu_m\to \cB \GL_n$, where the first map is the (second) projection. Up to replacing $\spec(R)$ with a ramified cover which extends its residue field, we can assume that $\cE_1$ and $\cE_2$ are isomorphic on the central fiber of $\cP^1\times \spec(R)$.

We argue that $\cE_1$ and $\cE_2$ are isomorphic on a formal neighbourhood of the central fiber of $\cP^1\times\spec(R)\to\spec(R)$, or in other terms, that they are isomorphic on any thickening of the central fiber of $\cP^1\times\spec(R)\to\spec(R)$. This will lead to the desired contradiction as, from Artin's approximation theorem, the two vector bundles $\cE_1$ and $\cE_2$ will be isomorphic on an \'etale neighbourhood of the central fiber, so on the geometric generic fiber. But $\cE_2$ comes from a map $\cP^1\to \cB\bmu_m\to \cB\GL_n$, so $\psi$ cannot be a twisted map generically, contradicting our assumptions.

We have that $\cH om(\cE_1,\cE_2)|_{\cP^1} = \cH om((\cE_1)|_{\cP^1},(\cE_2)|_{\cP^1})\cong (\cE_2)|_{\cP^1}\otimes (\cE_1)|_{\cP^1}^{\vee}$, where we identified $\cP^1$ with the central fiber of $\cP^1\times\spec(R)\to\spec(R)$. There are exact sequences as follows, where $\cG = (\cE_2)|_{\sP^1}\otimes (\cE_1)|_{\cP^2}^{\vee}$:
\[
0\to \cO_{\cP^1}\to \sigma_*\cO_{\bP^1}\to \cK\to 0 ,\]
\[0\to \cG\to \sigma_*\cO_{\bP^1}\otimes \cG =\sigma_*\sigma^*\cG \to \cK\otimes \cG\to 0.
\]
As the first sequence splits (using the trace map), the second sequence splits, so $\oH^1(\cP^1,\sigma_*\sigma^*\cG) = 0$ implies $\oH^1(\bP^1,\cG) = 0$. But $\cE_1$ comes from the map $\cP^1\to \cB\bmu_m\to \cB\GL_n$ so $\sigma^*\cE_1\cong \cO_{\bP^1}^{\oplus n}$ as $\bP^1\to \cP^1\to \cB\bmu_m$ factors via $\bP^1\to \spec(k_R)\to \cB\bmu_m$ where $k_R$ is the residue field of $R$. So $\oH^1(\cP^1,\sigma_*\sigma^*\cG) =\oH^1(\bP^1,\sigma^*\cG)= 0$.

{Now, let $\cP_n:=\spec(R/\pi^n)\times_{\spec(R)}\cC$ be the $n$-th infinitesimal neighbourhood of the central fiber, where $\pi$ is the uniformizer of $R$, and let $\cI_n$ is the ideal sheaf of $\cP^1_n$ in $\cP^1_{n+1}$ (observe that $\cI_n\cong \cO_{\cP^1}$). The cohomology long exact sequence associated to the following short exact sequence, together with $\oH^1(\cP^1,\cG) = 0$, guarantees that any isomorphism $(\cE_1)|_{\cP^1_n}\to (\cE_2)|_{\cP^1_n}$ (i.e. any section of $\cH om(\cE_1,\cE_2)|_{\cP^1_n}$ which restricts to an isomorphism on $\cP^1$) extends to an isomorphism $(\cE_1)|_{\cP^1_{n+1}}\to (\cE_2)|_{\cP^1_{n+1}}$:}
\[
0\to \cI_n:= \cO_{\cP^1}\to \cO_{\cP^1_{n+1}}\to \cO_{\cP^1_{n}}\to 0\]
\[0\to \cH om(\cE_1,\cE_2)|_{\cP^1}\to \cH om(\cE_1,\cE_2)|_{\cP^1_{n+1}}\to \cH om(\cE_1,\cE_2)|_{\cP^1_{n}}\to 0
\]

\textbf{Case $\cM \cong [\spec(A)/\GL_n]$.} This follows as in the proof of \Cref{prop degree 0 = factors through bmuk}: any morphism $\cP^1\to [\spec(A)/\GL_n]$ whose composition $\cP^1\to [\spec(A)/\GL_n]\to \cB\GL_n$ factors via $\cP^1\to \cB \bmu_m\to \cB\GL_n$, factors as $\cP^1\to [\spec(B)/\bmu_m]\to [\spec(A)/\GL_n]$. But as $\cP^1$ has a projective coarse moduli space and $[\spec(B)/\bmu_m]$ admits an affine coarse moduli space, any map $\cP^1\to [\spec(B)/\bmu_m]$ factors via $\cP^1\to \cB \bmu_\ell$ for some $\ell$.

\textbf{General case.} We use \cite{AHH}*{Proposition 2.10}. As the closed fiber of $\psi$ factors via $\cB\bmu_m\to \cM$, the composition of $\psi$ with $\cM\to M$ factors via a point. Then the map $\cP^1\times\spec(R)\to M$ factors as $\cP^1\times\spec(R)\to \spec(R) \to M$. From \cite{AHH}*{Proposition 2.10} we can take an \'etale neighbourhood of the form $[\spec(A)/\GL_n]$ such that $\spec(A^{\GL_n})\times_{M}\spec(R)$ is an \'etale neighbourhood of the closed point. So up to replacing $\spec(R)$ with a cover, we can assume that $\spec(R)\to M$ factors via $\spec(R)\to \spec(A^{\GL_n})\to M$, so $\psi$ factors via $[\spec(A)/\GL_n]\to \cM$. This case was treated before.
\end{proof}

\section{Two examples.}\label{section example}

In this section we study two particular examples: the case where the target stack is $\Theta:=[\bA^1/\Gm]$, and a special case of the case when the target stack is $\cB G$.

\subsection{Example with target stack $\cB G$} We consider in this example output of our algorithm, in the case when the generic fiber is smooth, there are no markings, and the target stack is $\cB G$ for a reductive group $G$.
This is particularly simple, as in this case the results of 
\Cref{sub:twisted blow-ups}, \Cref{subsection_md_blowups} and \Cref{subsection_iterated_twisted_blowup} are not needed.

The first step is to extend $C_\eta$ to a family of stable curves $C_R \to \spec(R)$ 
using the properness of $\overline{\cM}_g$. Then, we extend the $G$-torsor on the generic fiber of $C_R$ to the generic points of the irreducible 
components of the closed fiber of $C_R$, up to a ramified base change of $R$; this follows from \Cref{teorema fix}.
We then spread out such a $G$-torsor. This way, we get an open subset 
$U \subseteq C_R$ whose complement is a finite set of closed points 
$\{p_1,\ldots,p_m\}$, together with a $G$-torsor 
over $U$ extending the one over $C_\eta$.
Finally, we extend the $G$-torsor across the points 
$\{p_1,\ldots,p_m\}$ using Proposition 4.1. This last step adds stacky structure at the nodes. 

\subsection{Examples with target stack $[\bA^1/\Gm]$}

Recall that a morphism $X\to \Theta$ from a scheme $X$ is equivalent to the data of a line bundle $\cL$ on $X$, and a section $\cO_X\to \cL$. We give three examples of degenerations of maps to $\Theta$. In each one of these examples, we will have a DVR $R$ with generic point $\eta = \spec(K(R))$ and closed point $0$. We will start from a curve $C_\eta$ over $\eta$, with a map $f_\eta:C_\eta \to \Theta$ and we will try to extend it to $\spec(R)$.

\subsubsection{Example 1: smooth generic fiber}\label{subsection example 1}

Assume that we have a smooth curve $C_\eta$ over $\eta$ of genus at least 2, the generic point of $\spec(R)$, with a marked $K(R)$-point $p$. This corresponds to a section $\cO_{C_\eta}\to \cO_{C_\eta}(p)$, so to a morphism $C_\eta\to \Theta$. We want to extend this curve over $\spec(R)$.

As the good moduli space of $\Theta$ is the point, we first extend the stable map $C_\eta\to \spec(K(R))$ to $\spec(R)$ (which will also be the Kontsevich limit of the map $C_\eta \to \spec(k)$). Let $C_R$ be such extension, and assume that the closed fiber $C_0:= (C_R)_0$ is nodal. We aim at extending the map $f_\eta:C_\eta \to \Theta$ to a stacky modification of $C_R$. The first step of \Cref{theorem existence part val criterion} is to extend the morphism $f_\eta$ to the generic points of $C_0$.

In the proof of \Cref{theorem existence part val criterion} we use \Cref{teorema fix}, however in this case it is easier. We can take a proper closed subscheme $D_\eta\subseteq C_\eta$ containing $p$ such that on its complement the line bundle $\cO_{C_\eta}(p)$ is trivial. Let $D$ be the closure of $D_\eta$. The subscheme $D$ will not contain the generic points of $C_0$, and the map $\cO_{C_\eta}\to \cO_{C_\eta}(p)$ is isomorphic to $\cO_{C_\eta}\xrightarrow{\cdot 1} \cO_{C_\eta}$ on the complement of $D_\eta$, so it extends to the complement of $D$. 
Therefore, at the end of this step we have an extension of the map $f_\eta$ to a map defined the generic points of $C_0$, that then spreads out to an extension of the map defined over $\cU:= C_R\smallsetminus \{p_1,...,p_n\}$ for a finite set of closed points $p_i$.

Observe that there are multiple extensions one can pick. Indeed, for example, if $\Gamma$ is an irreducible component of $C_0$ which is Cartier in $C_R$, we can compose any map $\cO_{\cU}\to \cL$ with $\cO_\cU\to \cL\to \cL(n\Gamma)$ for every $n>0$. It is not hard to check, however, that this is the only freedom we have. Namely, once we fix $C_R$, two extensions of $f_\eta$ that are isomorphic in codimension one are isomorphic. And among all these extensions we have a "minimal" one: the one whose section does not vanish along any of the generic points of the fibers. We will choose this limit in the remaining part of this example.

We now have to extend the morphism $\cU\to \Theta$ to $C_R$, up to replacing some of the nodes with twisted nodes. For this step, it is useful to compare our limit with the corresponding limit in $\overline{\cM}_{g,1}$. Indeed, there are two cases. We denote by $(C^{DM},p^{DM})$ the limit of $(C_\eta, p)$ in $\overline{\cM}_{g,1}$. If $C^{DM}$ is stable, then $C_0 = C^{DM}$ and the line bundle with section is the Cartier divisor given by $p^{DM}$. If however $C^{DM}$ is not stable, since instead $(C^{DM}, p^{DM})$ is stable, the marked point is on a destabilizing $\bP^1$. In this case, the curve $C_0$ (which is obtained by contracting the destabilizing $\bP^1$ to a point) will have a twisted node $n$.
Indeed, there is no principal sheaf of ideals $I$ on $k[\![x,y]\!]/xy$ such that
$k[\![x,y]\!]/(xy,I)\cong k$. So on an \'etale chart around $n$, we will have a multisection with the action of $\bmu_n$ acting on the fibers of the map to $\spec(R)$. Moreover, the degeneration of $\cO_{C_\eta}(p)$ and its section will be a line bundle $\cL$ on $C_0$ with a section vanishing only at the twisted node.
\begin{center}
\includegraphics[scale=0.10]{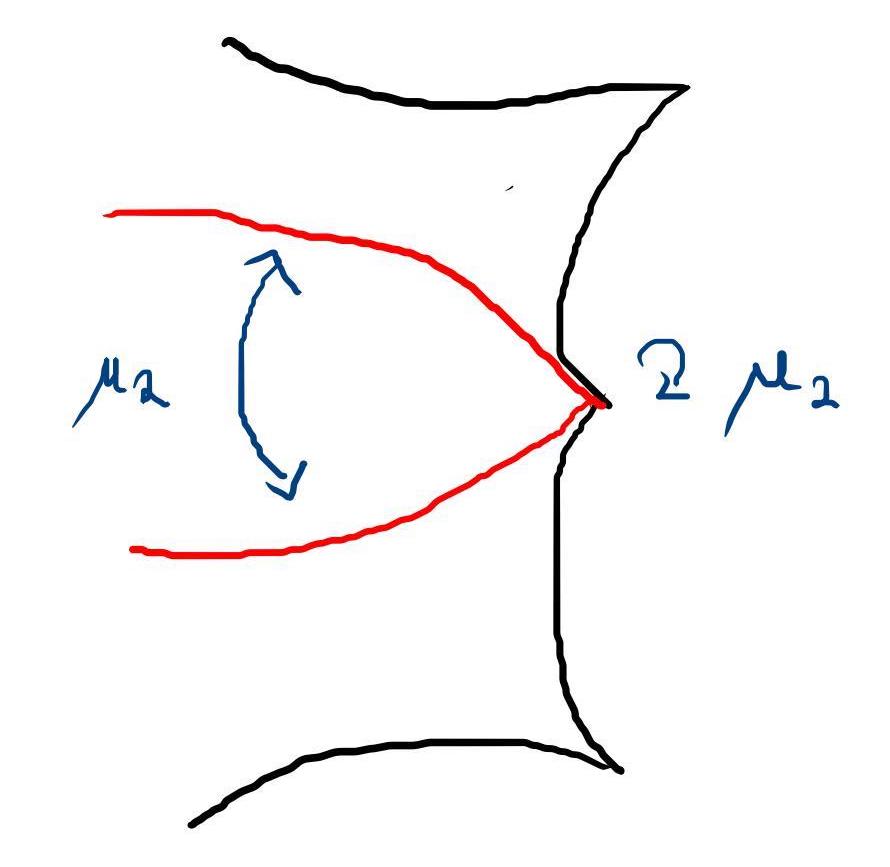} \\ Figure: The zero locus (in red) of $s \in H^0(\cL)$ on an \'etale chart around a node with a $\bmu_2$-stabilizer. In this picture, the singularity is $\spec(k[\![x,y]\!]^{\bmu_2})$ with $(x,y)\mapsto (-x,-y)$ and $V(s)=\{x-y=0\}$.
\end{center}

\subsubsection{Example 2: nodal but stable generic fiber}\label{subsection example 2}
Assume now that the generic fiber is as the closed fiber of \Cref{subsection example 1}. Namely, assume that we have a twisted curve $C$ whose coarse moduli space is a stable curve (i.e. it has no destabilizing $\bP^1$). Assume also that we have a line bundle $\cL$ with a section that vanishes only along a stacky node $n$, and assume that $C\smallsetminus \{n\}$ is connected. Let $k$ be the the order of the automorphism group of $n$. 

Let $C'$ the curve obtained by separating the node $n$. Namely, $C'$ is a partial normalization of $C$ which is isomorphic to $C$ away from $n$ and isomorphic to its normalization in a neighbourhood of $n$. Let also $i_1:n_1\to C'$ and $i_2:n_2\to C'$ be the inclusions of the nodal points in $C'$, i.e. the points lying over the node $n$ of $C$.
Then $C$ is isomorphic to the push-out $C' \cup_{n}
(n_1\cup n_2)$. This implies that the data of a map $ C \to \Theta$ is equivalent to the data of a map $ f':C'\to \Theta$, and an isomorphism $f'\circ i_1\to f'\circ i_2$.

Consider then the product family $f':C'\times\bA^1\to \bA^1$, let $\eta=\spec(k(t))$ and assume that we have a morphism $C_\eta\to \Theta$ given by a morphism $C'_\eta \to \Theta$ and the isomorphism $(f'\circ i_1)_\eta\to (f'\circ i_2)_\eta$ given by $t^m\in \Gm(k(t))$ for $m>0$. This data is equivalent to the data of a line bundle $L_\eta$ on $C'_\eta$ whose fiber $L_{\eta,n_1}\simeq k(t)$ over $n_1$ is glued to the fiber $L_{\eta,n_2}\simeq k(t)$ over $n_2$ via the isomorphism $k(t)\overset{\cdot t^m}{\to} k(t)$.

At this step we have to choose a $d$ such that the graded ring $k[t]$ (graded by the $\Gm$-action) is generated in degree at most $d$. Assume also that the stabilizer at $p$ has order $k$. Then, following the procedure of Theorem \ref{theorem existence part val criterion}, we need to perform a $(m,d)$-blow up on one of the branches of the node. The resulting exceptional component will have an exceptional (stacky) $\bP^1$ with two stacky points $p_1$ and $p_2$: the point $p_1$ has automorphism group
$\bmu_k$. The action is faithful, and assume that $\bmu_k$ acts on a generator $y$ of the stalk at $p_1$ by $\xi*y = \xi^{-1} y$. Whereas $p_2$ (which smooths along this degeneration) has automorphism group $\bmu_{k'd}$  where $k'=\frac{k}{mcd(k, d-1)}$. The zero of the section is still along the twisted node, and along the $\cP^1$. The degree of our line bundle along $\cP^1$ is $\frac{1}{dk}$. It can be understood as the line bundle of the divisor $\frac{1}{dk}[0] -\frac{1}{k}[0] + \frac{1}{k}[\infty]$ on a root-stack of $\bP^1$.
\begin{center}
\includegraphics[scale=0.189]{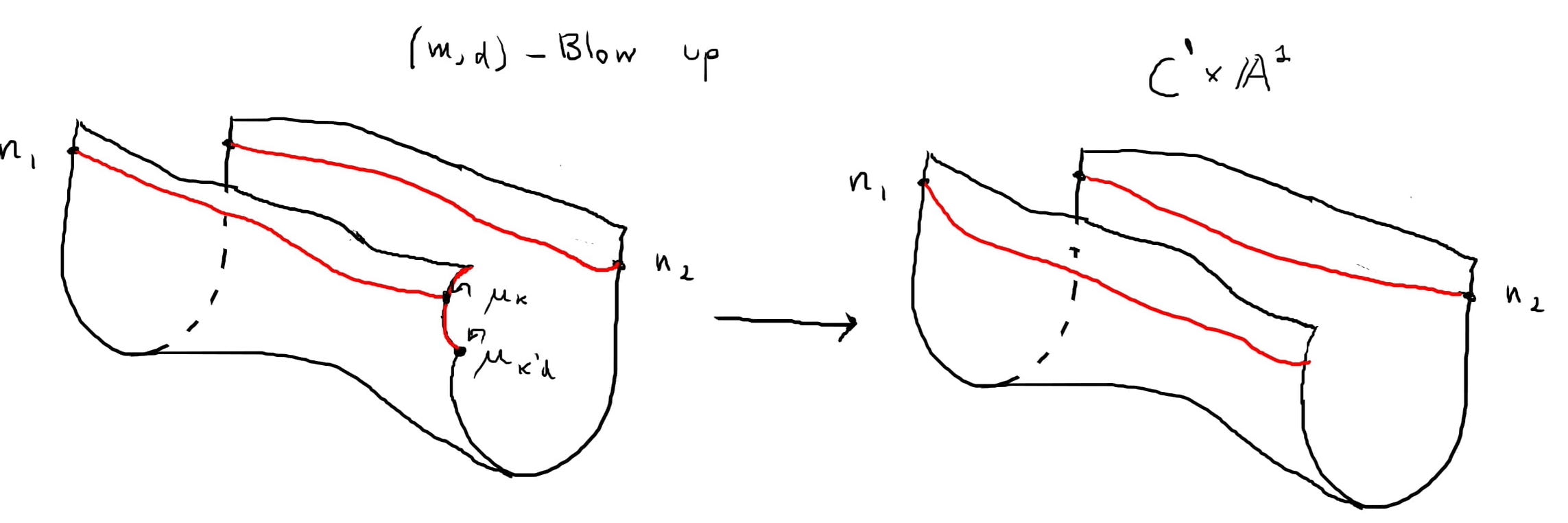} \\ Figure: quotient of the $(m,d)$-blow up. In red the zero locus of the section.
\end{center}

\subsubsection{Example 3: nodal unstable generic fiber}
For this last example, we assume that the generic fiber is as the closed fiber of \Cref{subsection example 2}.

 Namely, assume that we have a twisted curve $C$ whose coarse moduli space has a single destabilizing $\bP^1$. Let $\cP^1$ be the stacky-$\bP^1$ on $C$ that is destabilizing, with $n_1$ and $n_2$ the two nodes with automorphisms groups $\bmu_k$ and $\bmu_{k'd}$ respectively, where $k'=\frac{k}{mcd(k, d-1)}$.
 Assume that we have a line bundle $\cL$ with a section that vanishes \textit{only} along a stacky node $n_1\in \cP^1$, and along $\cP^1$.
 Assume that the restriction of $\cL$ to $\cP^1$ is $\cO_{\cP^1}(\frac{1}{dk}[0] -\frac{1}{k}[0] + \frac{1}{k}[\infty])$, and that $C\smallsetminus \{n_1\}$ is connected. Let $C'\cup \cP^1$ be the curve obtained by separating the nodes $n_1$ and $n_2$. Namely, $C'\cup \cP^1$ is a partial normalization of $C$ which is isomorphic to $C$ away from $n_1$ and $n_2$, and isomorphic to its normalization in a neighbourhood of $n_1$ and $n_2$. Let also $i_1:n_1\to C'$, $i_2:n_2\to C'$ and $j_1:n_1\to \cP^1$, $j_2:n_2\to \cP^1$ be the inclusions of the nodal points in $C'$ and $\cP^1$. The following picture helps keeping track of the notations \begin{center}
\includegraphics[scale=0.08]{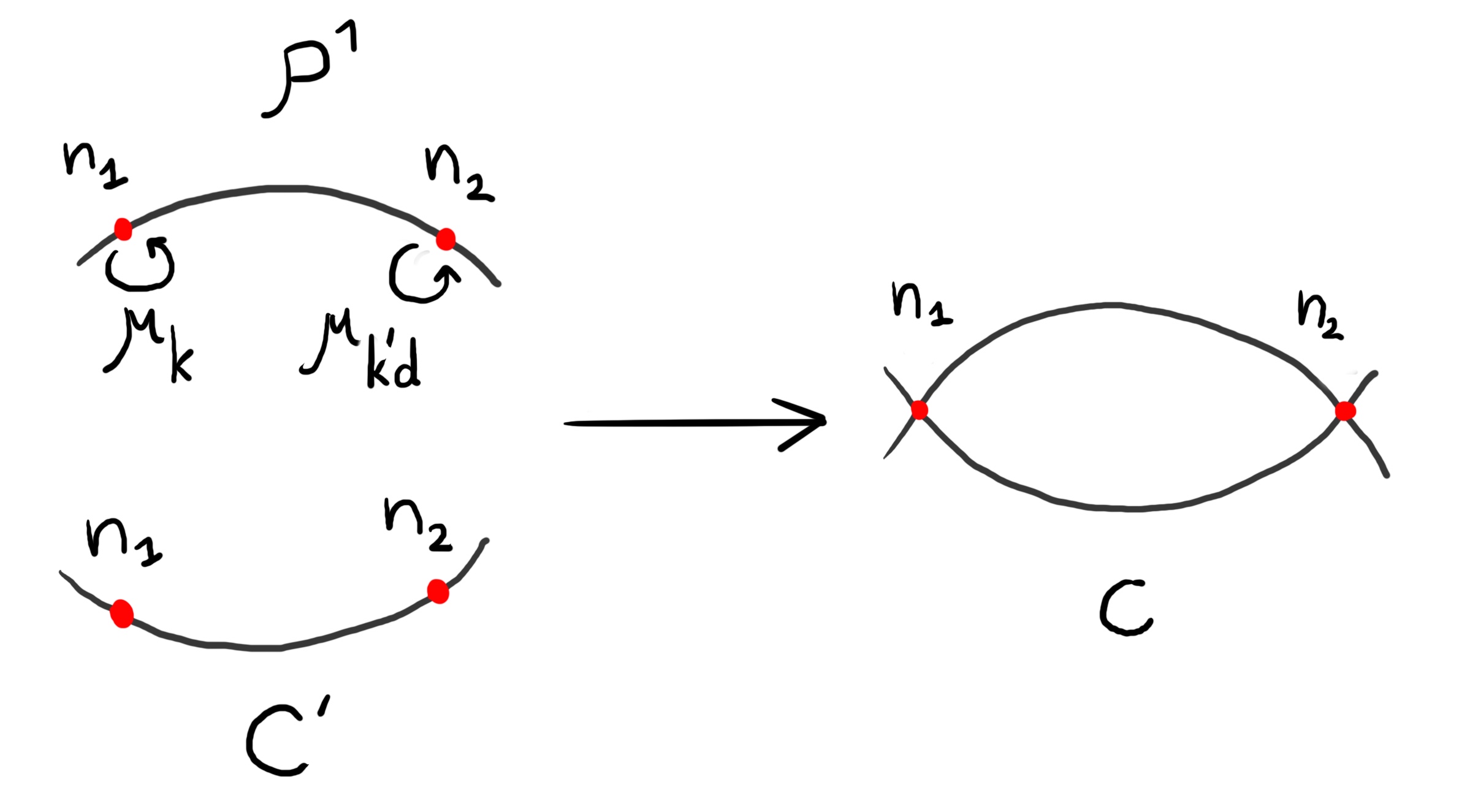} \\ Figure: the normalization of $C'$.
\end{center}
 
 As before $C$ is a push-out of a diagram, that is
$C= (C'\sqcup \cP^1) \cup_{(n_1\cup n_1)\cup (n_2\cup n_2)} (n_1\cup n_2).$
 Therefore the data of a map $ C \to \Theta$ is equivalent to two maps $ f':C'\to \Theta$ and $g':\cP^1\to \Theta$, and  two isomorphisms $f'\circ i_1\to f'\circ i_2$ and $g'\circ j_1\to g'\circ j_2$.
  
  Consider then the product family $f':(C'\cup \cP^1)\times\bA^1\to \bA^1$, let $\eta=\spec(k(t))$ and assume that we have a morphism $C_\eta\to \Theta$ given by two morphisms $f':C_\eta \to \Theta$ and $g':\cP^1_\eta \to \Theta$, and two isomorphism $(f'\circ i_1)_\eta\to (f'\circ i_2)_\eta$ and $(g'\circ j_1)_\eta\to (g'\circ j_2)_\eta$  given by $t^{m_1} \text{ and } t^{m_2}\in \Gm(k(t))$.

Now, the limit we will obtain will be easier than the one in \Cref{subsection example 2}:

\begin{bf}Claim.\end{bf} We can assume that $m_1\ge0$ and $m_2=0$.

Assuming the claim, we can perform an $(m_2,d)$-blow up along $n_2$, so the degeneration of $\cP^1$ along $\spec(R)$ will have two irreducible components: one is the exceptional divisor $\cE$ of the $(m_2,d)$-blow-up, and another component $\cD$. One can check (for example, using that the degree of $\cL$ on $\cE$ is $\frac{1}{kd}$ and \Cref{prop degree 0 = factors through bmuk}) that we can use Theorem \ref{theorem contracting stacky P1s} to contract $\cD$. So the limit will have a single destabilizing $\bP^1$.
\begin{center}
\includegraphics[scale=0.17]{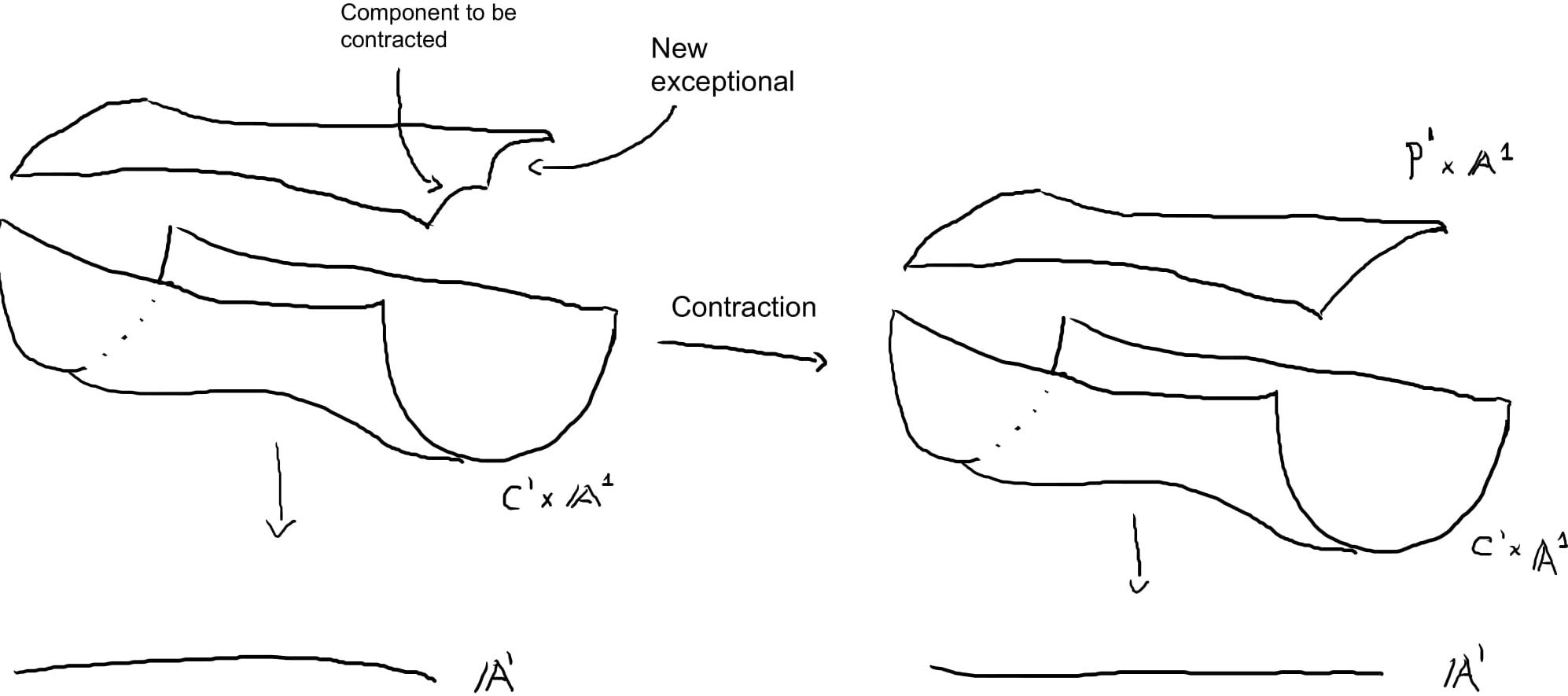} \\ Figure: Contraction of \Cref{theorem contracting stacky P1s}.
\end{center}
\begin{proof}[Proof of the Claim.] The map $\cP^1\to \Theta$ corresponds to the line bundle $\cG:=\cO_{\cP^1}(\frac{1}{dk}[0] -\frac{1}{k}[0] + \frac{1}{k}[\infty])$ on $\cP^1$ with the 0 section. We show that there is an automorphism of $(\cP^1,\cG)$ that acts via $t^{-m_2}$ on the fiber of $\cG$ at $n_2$ and by $\delta \gg0$ on the fiber of $\cG$ at $n_1$.

First, we can act on $\Gm(k(t))$ on the fibers of the line bundle. This will allow us to assume $m_2=0$: if we multiply each fiber by $\pi^{m_2}$ then the gluing map becomes 1. To introduce the next automorphism, we first present its analogue on $\cO_{\bP^1}(1)$ (i.e. in the case where there is no stacky structure).

We can construct the total space of $\cO_{\bP^1}(1)$ as the complement of the locus $(a,b)=(0,0)$ in $[\bA^3/\Gm]$, where the action of $\Gm$ is with weights $(1,1,1).$ Our automorphism will send $(a,b,v)\mapsto (t^\delta a, b, t^{\delta}v)$. This restricts to $\bP^1$ (the section $v=0$) to be the multiplication by $t^{\delta}$ on one chart of $\bP^1$, and by its inverse on the other chart. When $a=0$, the action on the fibers of the line bundle is the multiplication by $t^{\delta}$, and when $b=0$ it is the multiplication by $1$ (as $(t^\delta a, 0, t^{\delta}c)$ is in the same $\Gm$-orbit as $(a, 0, v)$). We will use an analogue of this action on $\cP^1$ to multiply the fibers of the line bundle by $t^\delta$ for $\delta \gg 0$ on $n_1$, and by $1$ on $n_2$.

Recall that $\cP^1$ is constructed as follows. First, we performed an $(m,d)$-blow up. The exceptional $\cD'$ is the $d$-th root stack of $\bP^1$ at $[0]$. Then we consider the $\bmu_k$-action on $\cD'$ by scaling, and we extend it to an action on $\cO_{\cD'}([0] \frac{1}{d})$.
The quotient is the root stack of $\bP^1$ at $[\infty]$ with root $k$, and at $[0]$ with root $dk$. The descended line bundle is $\cO_{\cD}([0] \frac{1}{kd})$ (it is the quotient of the line bundle $\cI$ of Lemma \ref{lemma blow up + covering stack does the job}, restricted to the exceptional). We denote this root stack by $\cD$, and we tensor the descended line bundle by $\cO_\cD(-\frac{1}{k}[0] + \frac{1}{k}[\infty])$ (this is the line bundle that does not extend in \Cref{lemma blow up + covering stack does the job}, restricted to the exceptional). Our final curve $\cP^1$ is the relative coarse moduli space of $\cP^1\to \cB\Gm$, where the map is given by the line bundle $\cO_\cD(\frac{1}{dk}[0] -\frac{1}{k}[0] + \frac{1}{k}[\infty])$. We will construct an automorphism of the line bundle $\cO_\cD(\frac{1}{dk}[0])$ which, when tensored by $\cO_\cD(-\frac{1}{k}[0] + \frac{1}{k}[\infty])$, will descend to the desired automorphism on the corresponding line bundle on $\cP^1$.

We can construct our automorphism \'etale locally. Namely, we construct a $\bmu_k$-invariant automorphism on $\cO_{\cD'}([0] \frac{1}{d})$. The total space of this line bundle is still a quotient of $\bA^3_{a,b,v}\smallsetminus \{(a,b)=(0,0)\}$ by $\Gm$, where the action has weights $(d,1,1)$. As before, we choose an automorphism of the form 
$$(a,b,v)\mapsto (t^{dk\delta}a, b, t^{k\delta}v) \text{ }\text{ }\text{ for }\delta \in \bZ.$$
The action on the fiber $b=0$ is trivial, and on the fiber $b=0$ is by the multiplication by $t^{k\delta}$. A generator $\xi$ of the group $\bmu_k$ sends $$\xi*(a,b,v)\mapsto (\xi a,b, \xi *v)$$
so the automorphism above commutes with this action (as $dk\delta$ is a multiple of $k$). 
\end{proof}

\bibliographystyle{amsalpha}
\bibliography{bibliography}

\end{document}